\documentclass[12pt]{amsart}
\usepackage[dvips,centering,includehead,width=15.1cm,height=22.7cm,
paperheight=23.8cm,paperwidth=17cm
]{geometry}

\usepackage[colorlinks=true]{hyperref}
\usepackage{ccaption}
\usepackage{pict2e}
\usepackage{enumitem}

\numberwithin{equation}{section}
\changecaptionwidth
\captionwidth{5.6in}

\usepackage{amsmath,amsthm,amsfonts,amssymb,amscd,euscript,tikz}
\usepackage{color}
\usepackage{graphics}
\newlength{\defbaselineskip}
\theoremstyle{plain}
\newtheorem{theorem}{Theorem}[section]
\newtheorem{proposition}[theorem]{Proposition}
\newtheorem{corollary}[theorem]{Corollary}
\newtheorem{lemma}[theorem]{Lemma}

\theoremstyle{definition}
\newtheorem{definition}[theorem]{Definition}

\newtheorem{remark}[theorem]{Remark}
\newtheorem{example}[theorem]{Example}

\newtheorem{observation}[theorem]{Observation}

\newcommand{\ZZ}{{\mathbb{Z}}}
\newcommand{\lefthalfarrow}{\succ\hspace{-2pt}\mapstochar}
\newcommand{\Hom}{{\operatorname{Hom}}}
\newcommand{\Disk}{{\mathbf{D}}}

\newcommand{\red}[1]{{\color{red} #1}}

\begin{document}

\title{Universal quivers}
\author{Sergey Fomin}
\address{Department of Mathematics, University of Michigan, Ann Arbor, MI 48109, USA}
\email{fomin@umich.edu}

\author{Kiyoshi Igusa}
\address{Department of Mathematics, Brandeis University, Waltham, MA 02454, USA}
\email{igusa@brandeis.edu}

\author{Kyungyong Lee}
\address{Department of Mathematics, University of Alabama, Tuscaloosa, AL 35487, USA; and School of Mathematics, Korea Institute for Advanced Study, Seoul 02455, Republic of Korea}
\email{kyungyong.lee@ua.edu; klee1@kias.re.kr}


\date{November 23, 2020. Revised March 6, 2021.} 


\thanks{Partially supported by NSF grants DMS-1664722 and DMS-1800207
as well as by the Simons Foundation.}


{\ }\vspace{-.2in} 

\begin{abstract}
We show that for any positive integer $n$, 
there exists a quiver $Q$ with $O(n^2)$ vertices and $O(n^2)$ edges
such that any quiver on $n$ vertices is a full subquiver
of a quiver mutation equivalent to $Q$. 
We generalize this statement to skew-symmetrizable matrices,
and obtain other related results. 
\end{abstract}

\maketitle

\vspace{-.15in} 

\section{Introduction}

Quivers and their mutations play a fundamental role in the combinatorial theory of cluster algebras~\cite{ca1, FWZ}. Among the various properties that a quiver may possess, one is especially interested in those which are both \emph{mutation-invariant} and \emph{hereditary}, i.e., preserved both by quiver mutations and by restriction to a full subquiver. 


A quiver~$Q$ (without frozen vertices) 
is called 
\emph{$n$-universal} if any quiver on $n$ vertices is a full subquiver of a quiver which is mutation equivalent to~$Q$. 
Our interest in this notion stems from the following simple remark. Let $P$ be a mutation-invariant and hereditary property of quivers. Then, in order to prove that any $n$-vertex quiver possesses property~$P$, it is sufficient to establish~$P$ for a single $n$-universal quiver. 

It is not at all obvious that $n$-universal quivers exist, even for small values of~$n$. 
(For any $n\ge 2$, there are infinitely many mutation classes of $n$-vertex quivers.) 
It~turns out that not only they do exist, but some of them are not too big. 
Our main result, presented in~Section~\ref{sec:constructing-universal-quivers}, is an explicit combinatorial construction that produces, for any~$n$, an $n$-universal quiver with $2n^2-n$ vertices and $7n^2-7n$ arrows. 

We demonstrate that some seemingly ``tame'' classes of quivers include universal quivers (which can be regarded as ``totally wild''). 
For any~$n$, there is an $n$-universal quiver on $O(n^2)$ vertices in which each vertex is incident to at most three arrows. 
We also construct a \emph{planar} $n$-universal quiver with $O(n^4)$ vertices and $O(n^4)$ arrows. 
This enables us to show in Section~\ref{sec:quivers-plabic} that any quiver can be embedded 
into a quiver mutation equivalent to a quiver of a \emph{plabic graph}. 
We note that the analogue of this statement for \emph{reduced} plabic graphs is false, 
see Remark~\ref{rem:reduced-plabic-not-universal}. 

In Section~\ref{sec:universal-matrices}, we extend our main result to skew-symmetrizable matrices and their mutations. The construction is largely the same, with some technical adjustments. 

Potential applications of these results and concepts are discussed 
in the last two sections. 
In Section~\ref{sec:hereditary+universal}, we clarify the relationship between hereditariness and universality, 
and illustrate it using such hereditary properties as 
Laurent positivity and sign-coherence of $c$-vectors. 
In Section~\ref{sec:categorification}, 
we show that the $2$-universal quivers constructed in 
Section~\ref{sec:constructing-universal-quivers}
do not allow any Hom-finite additive categorifications. 


\newpage

\section{Constructing universal quivers}
\label{sec:constructing-universal-quivers}

We begin by reviewing the relevant background on quiver mutations and mutation classes, generally following \cite[Sections 2.1, 2.6, 4.1]{FWZ}. 
Here and in Section~\ref{sec:universal-matrices}, 
we work with quivers without frozen vertices; so the definitions are slightly simpler than usual. 

\begin{definition}
\label{def:quiver-mutation}
A \emph{quiver} is a finite oriented graph with no loops or oriented $2$-cycles. The oriented edges of a quiver are called \emph{arrows}. For any vertex $v$ in a quiver~$Q$, the (quiver) \emph{mutation} $\mu_v$ transforms $Q$ into another quiver $Q'=\mu_v(Q)$ on the same vertex set, by performing the following three steps:
\begin{itemize}[leftmargin=.3in]
\item
for each oriented path $u\to v\to w$ passing through~$v$, add a new arrow
$u\to w$; 
\item
reverse all arrows incident to~$v$; 
\item
repeatedly remove oriented 2-cycles until there are none left.
\end{itemize}
\end{definition}

It is easy to see that quiver mutation is an involution: $\mu_v(\mu_v(Q))=Q$. 

A simple example is shown in Figure~\ref{fig:quiver-mutation}. 

\begin{figure}[ht]
\begin{center}
\begin{tikzpicture}
\coordinate (A) at (1.75,0);
\coordinate (A1) at (2.05,.08);
\coordinate (A2) at (2.05,-.08);
\coordinate (Ap) at (2,.25);
\coordinate (B) at (3.25,1.5);
\coordinate (B1) at (3.25,.3);
\coordinate (B2) at (3.25,-.3);
\coordinate (Bp) at (3.5,1.25);
\coordinate (Bl) at (3,1.25);
\coordinate (C) at (4.75,0);
\coordinate (C1) at (4.45,0.08);
\coordinate (C2) at (4.45,-.08);
\coordinate (Cl) at (4.5,.25);
\coordinate (L1) at (2.25,1);
\coordinate (L2) at (4.25,1);
\node at (A) {$u$};
\draw (A) circle [radius=0.2];
\draw[->](Ap)--(Bl);
\node at (B) {$v$};
\draw (B) circle [radius=0.2];
\draw[->](Bp)--(Cl);
\node at (C) {$w$};
\draw (C) circle [radius=0.2];
\draw[->](C1)--(A1);
\draw[->](C2)--(A2);
\end{tikzpicture}
\qquad\qquad
\begin{tikzpicture}
\coordinate (A) at (1.75,0);
\coordinate (A1) at (2.05,0);
\coordinate (A2) at (2.05,-.08);
\coordinate (Ap) at (2,.25);
\coordinate (B) at (3.25,1.5);
\coordinate (B1) at (3.25,.3);
\coordinate (B2) at (3.25,-.3);
\coordinate (Bp) at (3.5,1.25);
\coordinate (Bl) at (3,1.25);
\coordinate (C) at (4.75,0);
\coordinate (C1) at (4.45,0);
\coordinate (Cl) at (4.5,.25);
\coordinate (L1) at (2.25,1);
\coordinate (L2) at (4.25,1);
\node at (A) {$u$};
\draw (A) circle [radius=0.2];
\draw[<-](Ap)--(Bl);
\node at (B) {$v$};
\draw (B) circle [radius=0.2];
\draw[<-](Bp)--(Cl);
\node at (C) {$w$};
\draw (C) circle [radius=0.2];
\draw[->](C1)--(A1);
\end{tikzpicture}

\end{center}
\caption{Quiver mutation. Applying mutation $\mu_v$ to each of the two quivers shown above produces the other quiver.}
\vspace{-.1in}
\label{fig:quiver-mutation}
\end{figure}
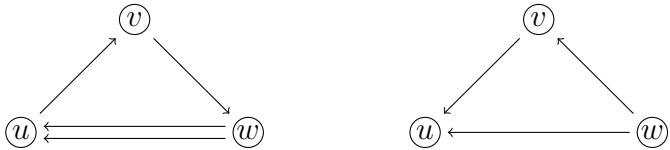

\begin{definition}
Two quivers $Q$ and $Q'$ are \emph{mutation equivalent} (denoted $Q\sim Q'$) if $Q$ can be transformed into a quiver isomorphic to~$Q'$ by a sequence of mutations. 

Mutation equivalence is an equivalence relation on the set of quivers (viewed up to isomorphism). Equivalence classes for this relation are called \emph{mutation classes}. We~denote by $[Q]$ the mutation class of a quiver~$Q$. Thus $[Q]$ consists of all (isomorphism types of) quivers mutation equivalent to~$Q$. 
\end{definition}

\begin{remark}[{cf.\ \cite[Problem 2.6.14]{FWZ}}]
There is no known algorithm for detecting whether a given pair of quivers (with the same number of vertices) are mutation equivalent to each other. Note that exhaustive search of all possible mutation scenarios does not qualify, since there is no known \emph{a~priori} upper bound on the depth of a successful search. 
\end{remark}

\begin{definition}
\label{def:full-subquiver}
Let $Q$ be a quiver, and $I$ a subset of the set of its vertices. The \emph{restriction} of $Q$ to~$I$, denoted~$Q_I$, is the induced subgraph of $Q$ supported on the vertex set~$I$. In other words, take $I$ as the new set of vertices, and keep all the arrows of~$Q$ which connect vertices in~$I$ to each other. Any quiver that can be obtained from~$Q$ by such a restriction is called a \emph{full subquiver} of~$Q$. 
\end{definition}

The following statement is well known, and easy to check.

\begin{lemma}
\label{lem:restriction-mutation}
Restriction commutes with mutation. More precisely, if $Q_I$ is a full sub\-quiver of~$Q$, and $v$ is a vertex in~$I$, then $\mu_v(Q_I)=(\mu_v(Q))_I$. 
\end{lemma}

Lemma~\ref{lem:restriction-mutation} implies that any quiver in 
$[Q_I]$ is a full subquiver of a quiver in~$[Q]$. 

\medskip\pagebreak[3]

We are now ready to introduce the main new concept of this paper. 

\begin{definition}
\label{def:n-universal}
Let $n\ge 2$ be an integer. 
A~quiver $Q$ is called \emph{$n$-universal} 
if every quiver on $n$ vertices is a full subquiver of a quiver mutation equivalent to~$Q$. 
\end{definition}

If a quiver is $n$-universal, then it is also $m$-universal for any $m\le n$. 

\begin{example}
The special case $n=2$ of Definition~\ref{def:n-universal} is the easiest one (but still important). A~quiver is $2$-universal if for any nonnegative integer~$k$, the given quiver can be mutated into a quiver containing two vertices connected by $k$ ``parallel'' arrows. 
Figure~\ref{fig:2-universal-quiver} shows an example of a 2-universal quiver on $3$~vertices. 
\end{example}

\begin{figure}[ht]
\begin{center}
\begin{tikzpicture}[scale=1,line width=0.6pt]
\node at (1.75,0) {1};
\draw (1.75,0) circle [radius=0.25];
\draw[->](2.1,0)--(2.9,0);
\node at (3.25,0) {2};
\draw (3.25,0) circle [radius=0.25];
\draw[->](3.6,0.1)--(4.4,0.1);
\draw[->](3.6,-0.1)--(4.4,-0.1);
\node at (4.75,0) {3};
\draw (4.75,0) circle [radius=0.25];
\end{tikzpicture}
\end{center}
\caption{A $2$-universal quiver. The mutation sequence $\mu_2,\mu_3,\mu_2,\mu_3,...$ produces every quiver on 2~vertices as a full subquiver supported either on the vertices 1~and~2, or on the vertices 1 and~3. }
\label{fig:2-universal-quiver}
\end{figure}
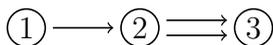

\begin{remark}
We are not aware of any algorithm for deciding whether a given quiver~$Q$ is $2$-universal, or whether $Q$ can be mutated to a quiver containing a pair of vertices connected by $k$ parallel arrows, cf.\ \cite[Remark~4.1.13]{FWZ}. For larger values of~$n$, the problem of detecting $n$-universality seems harder yet. 
\end{remark}

A quiver~$Q$ has \emph{finite mutation type} if its mutation class $[Q]$ is finite; in other words, there are finitely many (pairwise non-isomorphic) quivers mutation equivalent~to~$Q$. 

\begin{example}
\label{example:grid-quivers}
Any quiver of finite mutation type is \underbar{not} $2$-universal. One example of such a quiver is the \emph{grid quiver} $A_2\boxtimes A_5$, shown in Figure~\ref{fig:grid} on the left. On the other hand, the grid quiver $A_2\boxtimes A_6$ (see Figure~\ref{fig:grid} on the right) is $2$-universal, as it can be mutated to a quiver that restricts to the quiver from Figure~\ref{fig:2-universal-quiver}. 
\end{example}

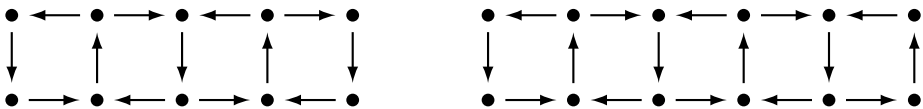
\begin{figure}[ht]
\begin{center}
\setlength{\unitlength}{3.2pt}
\begin{picture}(40,10)(0,0)
\multiput(0,0)(10,0){5}{\circle*{1.5}}
\multiput(0,10)(10,0){5}{\circle*{1.5}}
\thicklines
\multiput(2,0)(20,0){2}{\vector(1,0){6}}
\multiput(18,0)(20,0){2}{\vector(-1,0){6}}
\multiput(12,10)(20,0){2}{\vector(1,0){6}}
\multiput(8,10)(20,0){2}{\vector(-1,0){6}}

\multiput(10,2)(20,0){2}{\vector(0,1){6}}
\multiput(0,8)(20,0){3}{\vector(0,-1){6}}
\end{picture}
\qquad\qquad
\begin{picture}(50,10)(0,0)
\multiput(0,0)(10,0){6}{\circle*{1.5}}
\multiput(0,10)(10,0){6}{\circle*{1.5}}
\thicklines
\multiput(2,0)(20,0){3}{\vector(1,0){6}}
\multiput(18,0)(20,0){2}{\vector(-1,0){6}}
\multiput(12,10)(20,0){2}{\vector(1,0){6}}
\multiput(8,10)(20,0){3}{\vector(-1,0){6}}

\multiput(10,2)(20,0){3}{\vector(0,1){6}}
\multiput(0,8)(20,0){3}{\vector(0,-1){6}}
\end{picture}

\end{center}
\caption{The grid quivers $A_2\boxtimes A_5$ and $A_2\boxtimes A_6$.} 
\label{fig:grid}
\end{figure}

\begin{remark}
\label{rem:grid-not-3-universal}
The grid quiver $A_k\boxtimes A_\ell$ is \underbar{not} $3$-universal for any $k$ and~$\ell$. The reasons will be explained in Proposition~\ref{prop:grid-no-markov}. 
\end{remark}

A quiver is \emph{acyclic} if it does not contain oriented cycles. 
A quiver is called \emph{mutation-acyclic} if it can be mutated to an acyclic quiver.

\begin{remark}
\label{rem:acyclic-not-3-universal}
Any mutation-acyclic quiver $Q$ is \underbar{not} $3$-universal, for the following reason.  It has been shown in~\cite{BMR} that a full subquiver of a mutation-acyclic quiver is mutation-acyclic. If $Q$ were $3$-universal, then any $3$-vertex quiver would be a full subquiver of a quiver in~$[Q]$, and consequently would have to be mutation-acyclic. This contradicts the well known fact that some 3-vertex quivers are not mutation-acyclic. One easy example is the \emph{Markov quiver}, see Figure~\ref{fig:Markov-quiver}. 
\end{remark}

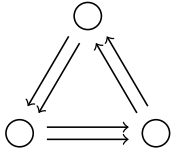
\begin{figure}[ht]
\begin{center}
\begin{tikzpicture}[scale=0.9,line width=0.6pt]
\draw (0,0) circle [radius=0.2];
\draw[->](0.4,0.09)--(1.6,0.09);
\draw[->](0.4,-0.09)--(1.6,-0.09);
\draw (2,0) circle [radius=0.2];
\draw (1,1.73) circle [radius=0.2];
\draw[<-](1.12,1.33)--(1.72,0.30);
\draw[<-](1.28,1.43)--(1.88,0.40);
\draw[->](0.88,1.33)--(0.28,0.30);
\draw[->](0.72,1.43)--(0.12,0.40);
\end{tikzpicture}
\vspace{-.1in}
\end{center}
\caption{The Markov quiver. }
\label{fig:Markov-quiver}
\end{figure}

The following is our main result. 

\begin{theorem}
\label{thm:main-thm-quivers}
For any integer $n\ge 2$, there exists an $n$-universal quiver with $2n^2-n$ vertices and $7n^2-7n$ arrows. 
\end{theorem}

The proof of Theorem~\ref{thm:main-thm-quivers} relies on two key lemmas.

\begin{lemma}
\label{lem:gluing}
Let $Q$ be a quiver with $k+2$ vertices, including vertices $u$ and~$v$, and $\ell$ arrows. Assume that for any nonnegative (resp., negative) integer~$m$, there exists a sequence of mutations, none of them at $u$ or~$v$, which transforms~$Q$ into a quiver that has exactly $m$ arrows directed from $u$ to~$v$ (resp., $|m|$ arrows directed from~$v$~to~$u$). \linebreak
Then for any $n$, there is an $n$-universal quiver with $n+k\binom{n}{2}$ vertices and $\ell\binom{n}{2}$ arrows. 
\end{lemma}

\begin{proof}
We construct the desired $n$-universal quiver $\overline Q$ by amalgamating (gluing) $\binom{n}{2}$ copies of the quiver~$Q$. 
We start with a totally disconnected quiver $Q_\bullet$ with $n$ vertices and no arrows. 
We then take $\binom{n}{2}$ disjoint copies of~$Q$, denoted $Q_{\{i,j\}}$, where $\{i,j\}$ runs over all $2$-element subsets of vertices in~$Q_\bullet$. 
In each of these copies, we mark the two vertices corresponding to the special vertices $u$ and $v$ in the original quiver~$Q$. 
Now, for each pair of vertices $i,j$ in~$Q_\bullet$, we glue $Q_{\{i,j\}}$ to~$Q_\bullet$ by identifying  $i$ and~$j$ with the marked vertices in~$Q_{\{i,j\}}$. (Any of the two possible identifications would work: either glue $i$ to~$u$ and $j$ to~$v$, or the other way around.) The resulting quiver $\overline Q$ will have $n+k\binom{n}{2}$ vertices and $\ell\binom{n}{2}$ arrows. It is not hard to see that $\overline Q$ is \hbox{$n$-universal}; in fact, any $n$-vertex quiver can be obtained by restricting a quiver in $[\overline Q]$ to the vertex set of~$Q_\bullet$. To see that, note that for any pair $i,j$ of vertices in~$Q_\bullet$, we can achieve the desired number of arrows between $i$ and~$j$ (with the desired direction) by performing mutations at the unmarked vertices in~$Q_{\{i,j\}}$. Moreover these mutations will not interact across different copies~$Q_{\{i,j\}}$ because there are no arrows connecting unmarked vertices from different copies to each other. 
\end{proof}

\pagebreak[3]

\begin{lemma}
\label{lem:6-vertex-quivers}
Each of the two quivers shown in Figures~\ref{fig:extended-somos-4}
and~\ref{fig:double-4-cycle}
satisfies the conditions in Lemma~\ref{lem:gluing}, for the vertices $u$ and~$v$ indicated in the drawings. 
\end{lemma}

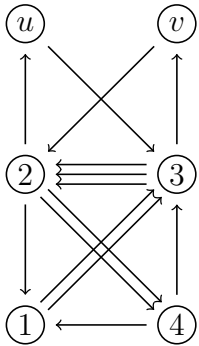
\begin{figure}[ht]
\begin{center}
\begin{tikzpicture}[scale=1,line width=0.6pt]
\node at (0,4) {$u$};
\draw (0,4) circle [radius=0.25];
\node at (2,4) {$v$};
\draw (2,4) circle [radius=0.25];
\node at (0,0) {1};
\draw (0,0) circle [radius=0.25];
\node at (0,2) {2};
\draw (0,2) circle [radius=0.25];
\node at (2,2) {3};
\draw (2,2) circle [radius=0.25];
\node at (2,0) {4};
\draw (2,0) circle [radius=0.25];

\draw[->](0,2.4)--(0,3.6);
\draw[->](2,2.4)--(2,3.6);
\draw[->](0.3,3.7)--(1.7,2.3);
\draw[<-](0.3,2.3)--(1.7,3.7);

\draw[<-](0,0.4)--(0,1.6);
\draw[->](2,0.4)--(2,1.6);
\draw[<-](0.4,0)--(1.6,0);

\draw[->](0.3,0.2)--(1.8,1.7);
\draw[->](0.2,0.3)--(1.7,1.8);
\draw[->](0.3,1.8)--(1.8,0.3);
\draw[->](0.2,1.7)--(1.7,0.2);

\draw[<-](0.4,2)--(1.6,2);
\draw[<-](0.4,2.13)--(1.6,2.13);
\draw[<-](0.4,1.87)--(1.6,1.87);
\end{tikzpicture}
\end{center}
\caption{The extended Somos-4 quiver.
}
\label{fig:extended-somos-4}
\end{figure}

\begin{figure}[ht]
\begin{center}
\begin{tikzpicture}[scale=1,line width=0.6pt]
\node at (1,2) {$u$};
\draw (1,2) circle [radius=0.25];
\node at (3,2) {$v$};
\draw (3,2) circle [radius=0.25];
\node at (0,0) {4};
\draw (0,0) circle [radius=0.25];
\node at (0,4) {1};
\draw (0,4) circle [radius=0.25];
\node at (4,4) {2};
\draw (4,4) circle [radius=0.25];
\node at (4,0) {3};
\draw (4,0) circle [radius=0.25];

\draw[->] (-0.08,0.4)--(-0.08,3.6);
\draw[->](0.08,0.4)--(0.08,3.6);
\draw[<-](3.92,0.4)--(3.92,3.6);
\draw[<-](4.08,0.4)--(4.08,3.6);
\draw[<-](0.4,-0.08)--(3.6,-0.08);
\draw[<-](0.4,0.08)--(3.6,0.08);
\draw[->](0.4,3.92)--(3.6,3.92);
\draw[->](0.4,4.08)--(3.6,4.08);

\draw[->](0.18,0.36)--(0.82,1.64);
\draw[<-](0.18,3.64)--(0.82,2.36);
\draw[->](3.82,0.36)--(3.18,1.64);
\draw[<-](3.82,3.64)--(3.18,2.36);

\draw[<-](0.3,0.2)--(2.7,1.8);
\draw[->](0.3,3.8)--(2.7,2.2);
\draw[->](3.7,3.8)--(1.3,2.2);
\draw[<-](3.7,0.2)--(1.3,1.8);
\end{tikzpicture}
\end{center}
\caption{Another quiver on 6 vertices satisfying the conditions in Lemma~\ref{lem:gluing}. }
\vspace{-.2in}
\label{fig:double-4-cycle}
\end{figure}
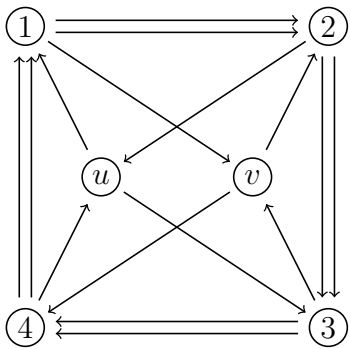

\begin{proof}
The proof relies on straightworward but tedious calculations, preferably done using one of the widely available software packages for quiver mutations~\cite{keller-applet, musiker-stump}. 

For the quiver shown in Figure~\ref{fig:extended-somos-4}, repeatedly apply the mutations~at the vertices labeled $1,2,3,4,1,2,3,4,1,2,3,4,\dots$
(in this order). Computations show that the resulting quivers have the following number of arrows directed from $v$ to~$u$:
\begin{equation}
\label{eq:0011111...}
0, 0, \underbrace{1,1,\dots,1,1}_{12},2,2,2,\underbrace{3,3,\dots,3,3}_{12},4,4,4,\underbrace{5,5,\dots,5,5}_{12},6,6,6,\dots
\end{equation}
(the leftmost entry corresponds to the original quiver). 
After 60 mutations, we recover the original quiver, with additional 8~arrows directed from $v$ to~$u$, and the process continues. To get the opposite orientation, i.e., arrows directed from $u$ to~$v$, take the original quiver and apply mutations at the vertices $4,3,2,1,4,3,2,1,4,3,2,1,\dots$ (in this order). 
In the resulting quivers, the number~of arrows directed from $u$ to~$v$ will again be given by the sequence~\eqref{eq:0011111...}, and we are~done. 
\pagebreak[3]

For the quiver shown in Figure~\ref{fig:double-4-cycle}, repeatedly apply the mutations~at the vertices labeled $1,3,2,4,1,3,2,4,1,3,2,4,\dots$ (in this order). The resulting quivers have the following number of arrows directed from $u$ to~$v$:
\begin{equation}
\label{eq:00112334...}
0,1,2,3,4,5,6,7,8,9,\dots
\end{equation}
(the first~0 corresponds to the original quiver). 
The quiver obtained after the first 4~mutations differs from the original quiver by having extra 4~arrows $u\!\longrightarrow \!v$, so the pattern continues. To get the opposite orientation, with arrows directed from $v$ to~$u$, apply mutations at the vertices $4,2,3,1,4,2,3,1,4,2,3,1,\dots$ (in this order) to the original quiver. 
In the resulting quivers, the number~of arrows $v\!\longrightarrow \!u$ will again be given by the sequence~\eqref{eq:00112334...}, and we are~done. 
\end{proof}

\pagebreak[3]

\begin{remark}
Restricting the quiver shown in Figure~\ref{fig:extended-somos-4} to the vertex set $\{1,2,3,4\}$ yields the \emph{Somos-4} quiver, see, e.g., \cite[Section~7.1]{marsh-lectures}. 
The same restriction applied to the quiver shown in Figure~\ref{fig:double-4-cycle} produces the product $A_1^{(1)}\boxtimes A_1^{(1)}$ of two Kronecker quivers of affine type~$A_1$. Both $4$-vertex quivers are closely related to discrete integrable systems. We~do not know whether this relationship is purely coincidental, or is a sign of some connection between universality and integrability. 
\end{remark}

\begin{proof}[Proof of Theorem~\ref{thm:main-thm-quivers}]
The theorem directly follows from Lemmas \ref{lem:gluing} and~\ref{lem:6-vertex-quivers}. 
The quiver shown in Figure~\ref{fig:extended-somos-4} has $6$~vertices and $14$~arrows. Applying (the construction in) Lemma~\ref{lem:gluing} with $k=4$ and $\ell=14$, we obtain an $n$-universal quiver with $n+4\binom{n}{2}=2n^2-n$ vertices and $14\binom{n}{2}=7n^2-7n$ arrows. 
An example for $n=3$ is shown in Figure~\ref{fig:3-universal}. 
\end{proof}

\begin{figure}[ht]
\begin{center}
\begin{tikzpicture}[scale=1,line width=0.6pt]
\draw (0,4) circle [radius=0.15];
\draw (2,4) circle [radius=0.25];
\node at (0,0) {1};
\draw (0,0) circle [radius=0.25];
\node at (0,2) {2};
\draw (0,2) circle [radius=0.25];
\node at (2,2) {3};
\draw (2,2) circle [radius=0.25];
\node at (2,0) {4};
\draw (2,0) circle [radius=0.25];

\draw[->](0,2.4)--(0,3.6);
\draw[->](2,2.4)--(2,3.6);
\draw[->](0.3,3.7)--(1.7,2.3);
\draw[<-](0.3,2.3)--(1.7,3.7);

\draw[<-](0,0.4)--(0,1.6);
\draw[->](2,0.4)--(2,1.6);
\draw[<-](0.4,0)--(1.6,0);

\draw[->](0.3,0.2)--(1.8,1.7);
\draw[->](0.2,0.3)--(1.7,1.8);
\draw[->](0.3,1.8)--(1.8,0.3);
\draw[->](0.2,1.7)--(1.7,0.2);

\draw[<-](0.4,2)--(1.6,2);
\draw[<-](0.4,2.13)--(1.6,2.13);
\draw[<-](0.4,1.87)--(1.6,1.87);

\draw[red] (2,4) circle [radius=0.15];
\draw[red] (3.73,5) circle [radius=0.25];
\draw[red] (5.46,6) circle [radius=0.25];
\draw[red] (1,5.73) circle [radius=0.25];
\draw[red] (2.73,6.73) circle [radius=0.25];
\draw[red] (4.46,7.73) circle [radius=0.25];
\node[red] at (3.73,5) {2};
\node[red] at (5.46,6) {1};
\node[red] at (2.73,6.73) {3};
\node[red] at (4.46,7.73) {4};
\draw[red,<-](2.346,4.2)--(3.384,4.8);
\draw[red,->](4.076,5.2)--(5.114,5.8);
\draw[red,<-](1.346,5.92)--(2.384,6.53);
\draw[red,<-](3.076,6.92)--(4.114,7.53);
\draw[red,->](2.1,4.36)--(2.63,6.37);
\draw[red,->](1.36,5.63)--(3.37,5.1);
\draw[red,->](2.93,6.39)--(3.53,5.346);
\draw[red,->](2.815,6.32)--(3.415,5.276);
\draw[red,->](3.045,6.46)--(3.645,5.416);
\draw[red,->](3.76,5.38)--(4.29,7.39);
\draw[red,->](3.90,5.34)--(4.43,7.35);
\draw[red,<-](3.07,6.56)--(5.08,6.03);
\draw[red,<-](3.11,6.7)--(5.12,6.17);
\draw[red,->](4.66,7.39)--(5.26,6.346);

\draw[cyan] (0,4) circle [radius=0.25];
\draw[cyan] (1,5.73) circle [radius=0.15];
\draw[cyan] (-3.46,6) circle [radius=0.25];
\draw[cyan] (-0.73,6.73) circle [radius=0.25];
\draw[cyan] (-2.46,7.73) circle [radius=0.25];
\draw[cyan] (-1.73,5) circle [radius=0.25];
\node[cyan] at (-1.73,5) {2};
\node[cyan] at (-3.46,6) {1};
\node[cyan] at (-0.73,6.73) {3};
\node[cyan] at (-2.46,7.73) {4};
\draw[cyan,<-](-0.346,4.2)--(-1.384,4.8);
\draw[cyan,->](-2.076,5.2)--(-3.114,5.8);
\draw[cyan,<-](0.654,5.92)--(-0.384,6.53);
\draw[cyan,<-](-1.076,6.92)--(-2.114,7.53);
\draw[cyan,->](-0.1,4.36)--(-0.63,6.37);
\draw[cyan,->](0.64,5.63)--(-1.37,5.1);
\draw[cyan,->](-0.93,6.39)--(-1.53,5.346);
\draw[cyan,->](-0.815,6.32)--(-1.415,5.276);
\draw[cyan,->](-1.045,6.46)--(-1.645,5.416);
\draw[cyan,->](-1.76,5.38)--(-2.29,7.39);
\draw[cyan,->](-1.90,5.34)--(-2.43,7.35);
\draw[cyan,<-](-1.07,6.56)--(-3.08,6.03);
\draw[cyan,<-](-1.11,6.7)--(-3.12,6.17);
\draw[cyan,->](-2.66,7.39)--(-3.26,6.346);

\end{tikzpicture}
\vspace{-.05in}
\end{center}
\caption{A $3$-universal quiver with 15 vertices and 42 arrows, obtained by gluing 3~copies of the extended Somos-4 quiver shown in Figure~\ref{fig:extended-somos-4}. Each copy can be affixed to the respective pair of  vertices in two different ways.}
\vspace{-.2in}
\label{fig:3-universal}
\end{figure}
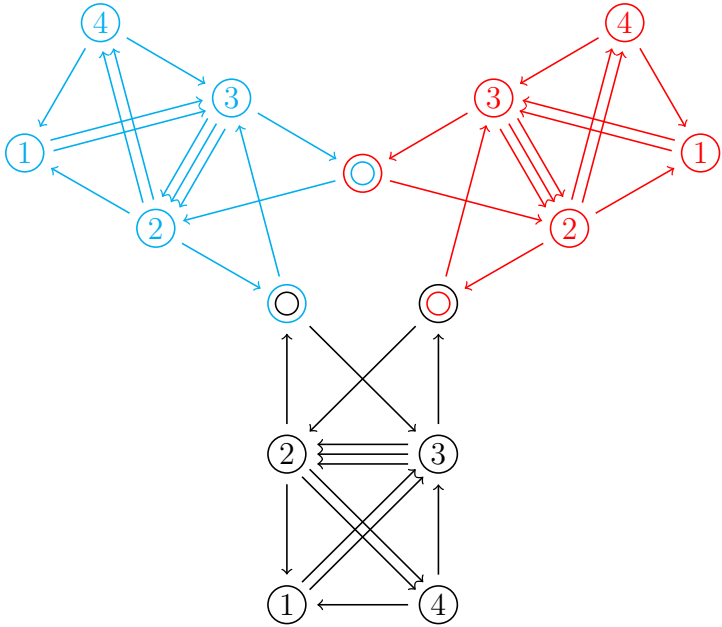

\begin{proposition}
\label{main_prop}
Let $Q$ be a quiver with $r$ arrows and $n$ vertices, none of which is a source or a sink. Then there exists a quiver $\widetilde{Q}$ such that
\begin{itemize}[leftmargin=.3in]
\item
$Q$ is a full subquiver of a quiver mutation equivalent to $\widetilde{Q}$; 
\item
$\widetilde{Q}$ has $2r-n$ vertices; 
\item
each vertex in $\widetilde{Q}$ is incident to at most 3 arrows.
\end{itemize}
\end{proposition}

\begin{proof}
Let us split each arrow $u\succ\hspace{-2pt}\rightarrow v$ in~$Q$ into two halves 
$u\!\lefthalfarrow\,$ and~$\,\mapsto \!v$. This produces $n$ fragments~$F_v$, one for each vertex~$v$ in~$Q$. We then replace each $F_v$ by a new fragment~$\widetilde F_v$, constructed as follows. Let $p=\operatorname{indeg}(v)$ and $q=\operatorname{outdeg}(v)$. Since $v$ is neither a source nor a sink, we have $p,q\ge 1$. 
The new fragment~$\widetilde F_v$ consists~of: 
\begin{itemize}[leftmargin=.3in]
\item
$p+q-1$ vertices, denoted 
$u_{p},\dots,u_2,u_1=w_1,w_2,\dots,w_{q}\,$;
\item
$p+q-2$ arrows $u_{p}\longrightarrow \cdots \longrightarrow u_2\longrightarrow u_1=w_1\longrightarrow w_2 \longrightarrow\cdots\longrightarrow w_{q}\,$;
\item 
$p$ incoming half-arrows: $\mapsto \!u_i$ for $2\le i\le p$, plus one extra half-arrow $\mapsto \!u_{p}\,$;
\item
$q$~outgoing half-arrows: $\,w_i\!\lefthalfarrow\,$ for $2\le i\le q$, plus one extra half-arrow $w_q\!\lefthalfarrow\,\,$. 
\end{itemize}
We then piece together the fragments $\widetilde F_v$, using the original quiver as a template: for each arrow $v\longrightarrow v'$ in~$Q$, we stitch one outgoing half-arrow from~$\widetilde F_v$ to one incoming half-arrow from~$\widetilde F_v'$. (The choices of particular half-arrows are immaterial.) 

The resulting quiver $\widetilde Q$ has 
$\sum_v (\operatorname{indeg}(v)+\operatorname{outdeg}(v)-1)=2r-n$ vertices. 
The total degree of each vertex in this quiver is either 2 or~3. 
Let us now mutate $\widetilde Q$ as follows: for each fragment $\widetilde F_v$ as above, mutate once at each of the vertices $u_2,\dots,u_{p},w_2,\dots,w_q$. (Mutations at different fragments commute.) 
After each mutation~$\mu_u$, remove the corresponding vertex~$u$, together with all arrows incident to~$u$.  The resulting quiver is canonically isomorphic to~$Q$. 
\end{proof}

\begin{corollary}
\label{main_thm2}
For any $n\ge 2$, there exists an $n$-universal quiver with fewer than $12n^2$ vertices, each of which is incident to at most 3 arrows. 
\end{corollary}

\begin{proof}
Apply Proposition~\ref{main_prop} to the $n$-universal quiver constructed in the proof of Theorem~\ref{thm:main-thm-quivers}.
\end{proof}

\begin{theorem}
\label{thm:planar-universal}
For any integer $n\ge 2$, there exists a planar $n$-universal quiver with $O(n^4)$ vertices and $O(n^4)$ arrows. 
\end{theorem}

\begin{proof}
As a warm-up, let us discuss the case $n=3$. 
The 3-universal~quiver shown in Figure~\ref{fig:3-universal} is planar. 
To see that, replace each edge of a regular triangle by a copy of the planar embedding of the extended Somos-4 quiver shown in Figure~\ref{fig:flattened-Somos-4}.
%
%

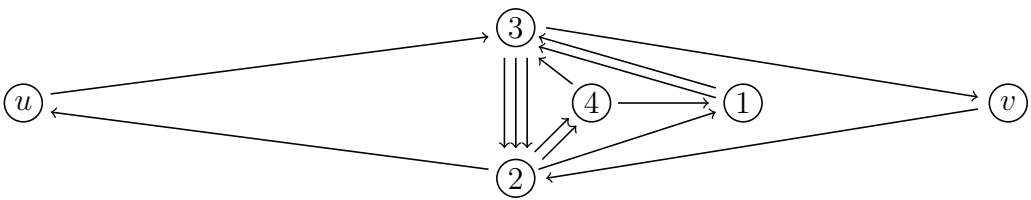
\begin{figure}[ht]
\begin{center}
\begin{tikzpicture}[scale=1,line width=0.6pt]
\node at (-3.5,0) {$u$};
\draw (-3.5,0) circle [radius=0.25];
\node at (3,1) {$3$};
\draw (3,1) circle [radius=0.25];
\node at (3,-1) {$2$};
\draw (3,-1) circle [radius=0.25];
\node at (4,0) {$4$};
\draw (4,0) circle [radius=0.25];
\node at (9.5,0) {$v$};
\draw (9.5,0) circle [radius=0.25];
\node at (6,0) {$1$};
\draw (6,0) circle [radius=0.25];
\draw[->](-3.14,0.12)--(2.64,0.88);
\draw[<-](-3.14,-0.12)--(2.64,-0.88);
\draw[->](3.4,1)--(9.1,0.08);
\draw[<-](3.4,-1)--(9.1,-0.08);
\draw[->](3.3,-0.88)--(5.64,-0.12);
\draw[<-](3.3,0.88)--(5.64,0.2);
\draw[<-](3.3,0.75)--(5.64,0.07);
\draw[->](4.35,0)--(5.55,0);
\draw[->](3,0.6)--(3,-0.6);
\draw[->](3.15,0.6)--(3.15,-0.6);
\draw[->](2.85,0.6)--(2.85,-0.6);
\draw[->](3.75,0.25)--(3.3,0.6);
\draw[<-](3.7,-0.2)--(3.25,-0.65);
\draw[<-](3.8,-0.3)--(3.35,-0.75);
\end{tikzpicture}
\vspace{-.1in}
\end{center}
\caption{A planar rendition of the extended Somos-4 quiver. }
\vspace{-.2in}
\label{fig:flattened-Somos-4}
\end{figure}

\pagebreak[3]

For general~$n$, we will need an additional trick. 

\begin{lemma}
\label{lem:resolving-crossings}
Let $Q$ be a quiver with $k$ vertices and $\ell$ arrows.
Suppose $Q$ can be drawn on the plane so that its arrows cross at $m$ points, two arrows at each crossing. Then there exists a planar quiver $Q'$ with $k+5m$ vertices and $\ell+8m$ arrows which can be mutated into a quiver that has $Q$ as a full subquiver. 
\end{lemma}

\begin{proof}
Let us replace each of the $m$ crossings by a (small) fragment of the kind shown in Figure~\ref{fig:replacing-a-crossing}. 
This produces a planar quiver~$Q'$ with the required parameters. Now perform the following operations for each of the new fragments: 
mutate at~$e$; remove~$e$; 
mutate at $a$ and~$b$; remove $a$ and~$b$;
mutate at $c$ and~$d$; remove $c$ and~$d$. 
This recovers the original quiver~$Q$. 
\end{proof}

\begin{figure}[ht]
\begin{center}
\vspace{-.1in}
\begin{tikzpicture}[scale=1,line width=0.6pt]
\draw[->](-2.6,0)--(2.6,0);
\draw[->](0,-2.6)--(0,2.6);
\node at (4,0) {$\leadsto$};
\end{tikzpicture}
\qquad  
\begin{tikzpicture}[scale=1,line width=0.6pt]
\node at (0,0) {$e$};
\draw (0,0) circle [radius=0.25];
\node at (1.5,0) {$c$};
\draw (1.5,0) circle [radius=0.25];
\node at (-1.5,0) {$a$};
\draw (-1.5,0) circle [radius=0.25];
\node at (0,1.5) {$b$};
\draw (0,1.5) circle [radius=0.25];
\node at (0,-1.5) {$d$};
\draw (0,-1.5) circle [radius=0.25];
\draw[<-](-1.25,0.25)--(-0.25,1.25);
\draw[<-](0.25,-1.25)--(1.25,-0.25);
\draw[->](-1.1,0)--(-0.4,0);
\draw[->](-2.6,0)--(-1.9,0);
\draw[->](0.4,0)--(1.1,0);
\draw[->](1.9,0)--(2.6,0);

\draw[->](0,-1.1)--(0,-0.4);
\draw[->](0,-2.6)--(0,-1.9);
\draw[->](0,0.4)--(0,1.1);
\draw[->](0,1.9)--(0,2.6);
\end{tikzpicture}
\vspace{-.15in}
\end{center}
\caption{Replacing a crossing by a planar fragment. }
\vspace{-.2in}
\label{fig:replacing-a-crossing}
\end{figure}
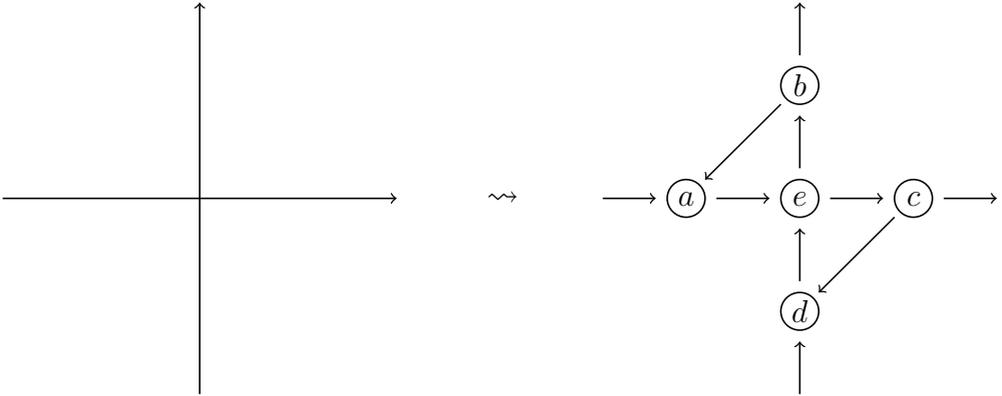

We are now ready to prove Theorem~\ref{thm:planar-universal}. 
Let $v_1,\dots,v_n$ be the vertices of a convex $n$-gon on the real plane, chosen so that no three diagonals of it are concurrent. 
As in the proof of Theorem~\ref{thm:main-thm-quivers}, we glue a copy $E_{ab}$ of the extended Somos-4 quiver between each pair of vertices $\{v_a, v_b\}$. 
In doing so, we flatten each subquiver $E_{ab}$ as shown in Figure~\ref{fig:very-flattened-Somos-4}, making sure that for any $a<b<c<d$, the drawings of $E_{ac}$ and $E_{bd}$ cross at exactly four points. The resulting quiver~$Q$ is isomorphic to the quiver constructed in the proof of Theorem~\ref{thm:main-thm-quivers}, so is $n$-universal, with $k=2n^2-n$ vertices and $\ell=7n^2-7n$ arrows. The drawings of the arrows of~$Q$ intersect at $m=4\binom{n}{4}$ points. 
Applying Lemma~\ref{lem:resolving-crossings}, we obtain a quiver with $2n^2-n+20\binom{n}{4}\sim\frac56 n^4$ vertices and $7n^2-7n+32\binom{n}{4}\sim \frac43 n^4$ arrows. This quiver is $n$-universal since it can be mutated into a quiver that restricts to~$Q$. 
\end{proof}

\begin{figure}[ht]
\begin{center}
\begin{tikzpicture}[scale=0.4,line width=0.6pt]
\draw (-2,0) circle [radius=0.2];
\draw (3,1) circle [radius=0.2];
\draw (3,-1) circle [radius=0.2];
\draw (4,0) circle [radius=0.2];
\draw (6,0) circle [radius=0.2];
\draw[->](-1.64,0.12)--(2.64,0.88);
\draw[<-](-1.64,-0.12)--(2.64,-0.88);
\draw[->](3.3,-0.88)--(5.64,-0.12);
\draw[<-](3.3,0.88)--(5.64,0.2);
\draw[<-](3.3,0.75)--(5.64,0.07);
\draw[->](4.35,0)--(5.55,0);
\draw[->](3,0.6)--(3,-0.6);
\draw[->](3.15,0.6)--(3.15,-0.6);
\draw[->](2.85,0.6)--(2.85,-0.6);
\draw[->](3.75,0.25)--(3.3,0.6);
\draw[<-](3.7,-0.2)--(3.25,-0.65);
\draw[<-](3.8,-0.3)--(3.35,-0.75);
\draw[<-](3.35,-1.1)--(32.6,-0.15);
\draw[->](3.35,1.1)--(32.6,0.15);
\draw (33,0) circle [radius=0.2];
\end{tikzpicture}
\vspace{-.1in}
\end{center}
\caption{A deformation of the drawing in Figure~\ref{fig:flattened-Somos-4}.}
\vspace{-.2in}
\label{fig:very-flattened-Somos-4}
\end{figure}
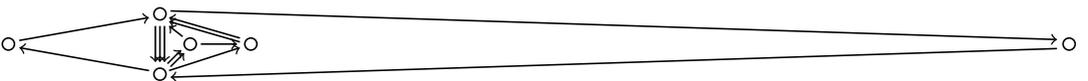

\begin{remark}
The asymptotics obtained above can be improved to $\sim\frac5{16} n^4$ vertices and $\sim \frac12 n^4$ arrows by using the drawings of the complete graph~$K_n$ with the smallest known number of crossings; see, e.g., \cite[Section~13.2]{szekely}. 
\end{remark}

\section{Quivers of plabic graphs}
\label{sec:quivers-plabic}

Theorem~\ref{thm:planar-universal} has a corollary concerning quivers
associated with Postnikov's plabic graphs~\cite{postnikov}. 
In order to~state this corollary, we recall the requisite background,
adapting it to our current purposes and borrowing examples from
\cite[Section~6]{FPST}. 

\begin{definition}
\label{def:plabic}
Let $P$ be a finite connected planar graph~$P$ properly embedded
(as a $1$-dimensional cell complex) into a closed disk~$\Disk$. 
We call $P$ a (trivalent) \emph{plabic} (=planar bicolored) \emph{graph}
if the following conditions are satisfied: 
\begin{itemize}[leftmargin=.3in]
\item
each interior vertex of~$P$ is colored in one of the two colors, either black or white; 
\item
each vertex of $P$ lying in the interior of~$\Disk$ is trivalent 
(i.e., has degree~$3$);
\item 
each vertex of $P$ lying on the boundary~$\partial\Disk$ is univalent 
(i.e., has degree~$1$). 
\end{itemize}
Plabic graphs are viewed up to isotopy of the ambient disk~$\Disk$, 
and up to simultaneous reversal of the colors of all vertices. 
%
The \emph{move equivalence} is an equivalence relation on plabic graphs 
generated by the following \emph{local moves}: 
\begin{itemize}[leftmargin=.3in]
\item The \emph{flip} move replaces two adjacent trivalent vertices of
  the same color with two other vertices of the same color, connected
  in a different way:
  
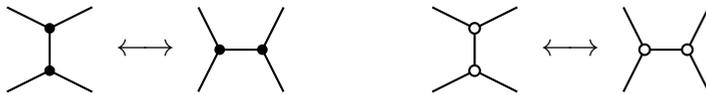
\begin{center}
\setlength{\unitlength}{0.8pt}
\begin{picture}(40,45)(0,0)
\thicklines
\put(0,0){\line(2,1){20}}
\put(20,30){\line(2,1){20}}
\put(0,40){\line(2,-1){20}}
\put(20,10){\line(2,-1){20}}
\put(20,10){\line(0,1){20}}
\put(20,10){\circle*{5}}
\put(20,30){\circle*{5}}
\end{picture}
\begin{picture}(40,40)(0,0)
\put(20,20){\makebox(0,0){$\longleftrightarrow$}}
\end{picture}
\begin{picture}(40,40)(0,0)
\thicklines
\put(0,0){\line(1,2){10}}
\put(30,20){\line(1,2){10}}
\put(40,0){\line(-1,2){10}}
\put(10,20){\line(-1,2){10}}
\put(10,20){\line(1,0){20}}
\put(10,20){\circle*{5}}
\put(30,20){\circle*{5}}
\end{picture}
\begin{picture}(60,40)(0,0)
\end{picture}
\begin{picture}(40,40)(0,0)
\thicklines
\put(0,0){\line(2,1){18}}
\put(40,40){\line(-2,-1){18}}
\put(0,40){\line(2,-1){18}}
\put(40,0){\line(-2,1){18}}
\put(20,12.5){\line(0,1){15}}
\put(20,10){\circle{5}}
\put(20,30){\circle{5}}
\end{picture}
\begin{picture}(40,40)(0,0)
\thicklines
\put(20,20){\makebox(0,0){$\longleftrightarrow$}}
\end{picture}
\begin{picture}(40,40)(0,0)
\thicklines
\put(0,0){\line(1,2){9}}
\put(40,40){\line(-1,-2){9}}
\put(40,0){\line(-1,2){9}}
\put(0,40){\line(1,-2){9}}
\put(12.5,20){\line(1,0){15}}
\put(10,20){\circle{5}}
\put(30,20){\circle{5}}
\end{picture}
\end{center}
\item The \emph{square} move switches the colors on a 4-cycle of
  vertices of alternating colors:
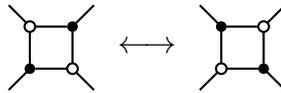
\begin{center}
\setlength{\unitlength}{0.8pt}
\begin{picture}(40,45)(0,0)
\thicklines
\put(0,0){\line(1,1){8.5}}
\put(40,40){\line(-1,-1){8.5}}
\put(0,40){\line(1,-1){8.5}}
\put(40,0){\line(-1,1){8.5}}
\put(10,12.5){\line(0,1){15}}
\put(30,12.5){\line(0,1){15}}
\put(12.5,10){\line(1,0){15}}
\put(12.5,30){\line(1,0){15}}
\put(10,10){\circle*{5}}
\put(10,30){\circle{5}}
\put(30,10){\circle{5}}
\put(30,30){\circle*{5}}
\end{picture}
\begin{picture}(40,40)(0,0)
\thicklines
\put(20,20){\makebox(0,0){$\longleftrightarrow$}}
\end{picture}
\begin{picture}(40,40)(0,0)
\thicklines
\put(0,0){\line(1,1){8.5}}
\put(40,40){\line(-1,-1){8.5}}
\put(0,40){\line(1,-1){8.5}}
\put(40,0){\line(-1,1){8.5}}
\put(10,12.5){\line(0,1){15}}
\put(30,12.5){\line(0,1){15}}
\put(12.5,10){\line(1,0){15}}
\put(12.5,30){\line(1,0){15}}
\put(10,10){\circle{5}}
\put(10,30){\circle*{5}}
\put(30,10){\circle*{5}}
\put(30,30){\circle{5}}
\end{picture}
\end{center}
\end{itemize}
(This move requires a seldom relevant restriction on the faces surrounding the square face, 
cf.\ \cite[Restriction 6.3]{FPST}.) 
\end{definition}

\begin{definition}
\label{def:Q(P)}
The quiver~$Q(P)$ associated with a plabic graph~$P$ is constructed as follows.
Place a vertex of~$Q(P)$ into each bounded face of~$P$.
For each edge~$e$ in~$P$ that connects vertices of different colors,
draw an arrow across~$e$ connecting the vertices of~$Q(P)$ located inside the faces 
on the two sides of~$e$. (We assume that these faces are bounded and distinct.) 
Orient this arrow so that the black endpoint of~$e$ appears on its right as one moves
in the chosen direction. 
Then remove oriented cycles of length~2, if any. 
See Figure~\ref{fig:quivers-plabic-graphs}.
\end{definition}


\begin{figure}[ht]
\begin{center}
\setlength{\unitlength}{0.8pt}
\begin{picture}(100,40)(0,-10)
\thicklines
\put(10,0){\circle*{5}}
\put(30,0){\circle{5}}
\put(70,0){\circle*{5}}
\put(90,0){\circle{5}}
\put(10,20){\circle{5}}
\put(30,20){\circle*{5}}
\put(50,20){\circle*{5}}
\put(70,20){\circle{5}}
\put(10,40){\circle*{5}}
\put(50,40){\circle{5}}
\put(70,40){\circle*{5}}
\put(90,40){\circle{5}}

\put(0,0){\line(1,0){27.5}}
\put(32.5,0){\line(1,0){55}}
\put(92.5,0){\line(1,0){7.5}}
\put(12.5,20){\line(1,0){55}}
\put(0,40){\line(1,0){47.5}}
\put(52.5,40){\line(1,0){35}}
\put(92.5,40){\line(1,0){7.5}}

\put(10,2.5){\line(0,1){15}}
\put(10,22.5){\line(0,1){15}}
\put(30,17.5){\line(0,-1){15}}
\put(50,22.5){\line(0,1){15}}
\put(70,2.5){\line(0,1){15}}
\put(70,22.5){\line(0,1){15}}
\put(90,2.5){\line(0,1){35}}

\put(20,10){\red{\circle*{5}}}
\put(61,10){\red{\circle*{5}}}

\put(20,30){\red{\circle*{5}}}
\put(61,30){\red{\circle*{5}}}
\put(83,20){\red{\circle*{5}}}

\put(24,30){\red{\vector(1,0){33}}}
\put(56,10){\red{\vector(-1,0){33}}}
\put(79,22){\red{\vector(-2.3,1){15}}}
\put(20,14){\red{\vector(0,1){13}}}
\put(61,26){\red{\vector(0,-1){13}}}
\put(64,12){\red{\vector(2.3,1){15}}}

\end{picture}
\qquad
\setlength{\unitlength}{0.8pt}
\begin{picture}(120,58)(-50,0)
\thinlines
\thicklines
\put(10,22.5){\line(0,1){15}}
\put(10,2.5){\line(0,1){15}}
\put(10,42.5){\line(0,1){15}}
\put(30,22.5){\line(0,1){15}}
\qbezier(32.5,20)(50,20)(50,30)
\qbezier(32.5,40)(50,40)(50,30)
\qbezier(12.5,0)(70,0)(70,30)
\qbezier(7.5,0)(-20,0)(-20,30)
\qbezier(7.5,60)(-20,60)(-20,30)
\qbezier(12.5,60)(70,60)(70,30)
\put(-20,30){\line(-1,0){20}}
\put(70,30){\line(1,0){20}}

\put(12.5,20){\line(1,0){15}}
\put(12.5,40){\line(1,0){15}}
\put(10,0){\circle{5}}
\put(10,60){\circle*{5}}
\put(10,20){\circle*{5}}
\put(10,40){\circle{5}}
\put(30,20){\circle{5}}
\put(30,40){\circle*{5}}
\put(-20,30){\circle*{5}}
\put(70,30){\circle*{5}}

\put(-5,40){\red{\circle*{5}}}
\put(19,30){\red{\circle*{5}}}
\put(40,30){\red{\circle*{5}}}
\put(30,50){\red{\circle*{5}}}

\put(36,30){\red{\vector(-1,0){13}}}
\put(32,47){\red{\vector(1,-2){7}}}
\put(22,33){\red{\vector(1,2){7}}}
\put(20,34){\red{\vector(1,2){7}}}
\put(-1,38){\red{\vector(3,-1){18}}}
\put(27,49){\red{\vector(-7,-2){29}}}
\put(26,51){\red{\vector(-7,-2){29}}}

\end{picture}
\end{center}
\caption{Quivers associated with plabic graphs.
Double arrows arise when two faces of a
plabic graph share two disconnected boundary segments.} 
\label{fig:quivers-plabic-graphs}
\vspace{-.2in}
\end{figure}
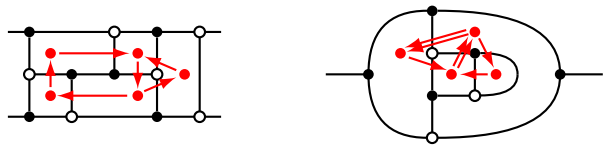

\begin{remark}
There is a version of Definitions~\ref{def:plabic}--\ref{def:Q(P)} 
in which the interior vertices of a plabic graph are not required to be trivalent. 
The class of quivers associated with these more general plabic graphs 
is exactly the same as in the trivalent setting. 
\end{remark}

The following observation is implicit in Postnikov's original work~\cite{postnikov}.

\begin{proposition} 
\label{pr:plabic-vs-quivers}
If two plabic graphs are move equivalent to each other, 
then their associated quivers are mutation equivalent. 
\end{proposition}

Proposition~\ref{pr:plabic-vs-quivers} is illustrated in 
Figure~\ref{fig:plabic-moves-are-mutations}.


\begin{figure}[ht]
\begin{center}
\setlength{\unitlength}{0.8pt}
\begin{picture}(100,40)(0,0)
\thicklines
\put(10,0){\circle*{5}}
\put(30,0){\circle{5}}
\put(70,0){\circle*{5}}
\put(90,0){\circle{5}}
\put(10,20){\circle{5}}
\put(30,20){\circle*{5}}
\put(50,20){\circle*{5}}
\put(70,20){\circle{5}}
\put(10,40){\circle*{5}}
\put(50,40){\circle{5}}
\put(70,40){\circle*{5}}
\put(90,40){\circle{5}}

\put(0,0){\line(1,0){27.5}}
\put(32.5,0){\line(1,0){55}}
\put(92.5,0){\line(1,0){7.5}}
\put(12.5,20){\line(1,0){55}}
\put(0,40){\line(1,0){47.5}}
\put(52.5,40){\line(1,0){35}}
\put(92.5,40){\line(1,0){7.5}}

\put(10,2.5){\line(0,1){15}}
\put(10,22.5){\line(0,1){15}}
\put(30,17.5){\line(0,-1){15}}
\put(50,22.5){\line(0,1){15}}
\put(70,2.5){\line(0,1){15}}
\put(70,22.5){\line(0,1){15}}
\put(90,2.5){\line(0,1){35}}

\put(20,10){\red{\circle*{5}}}
\put(61,10){\red{\circle*{5}}}

\put(20,30){\red{\circle*{5}}}
\put(61,30){\red{\circle*{5}}}
\put(83,20){\red{\circle*{5}}}

\put(24,30){\red{\vector(1,0){33}}}
\put(56,10){\red{\vector(-1,0){33}}}
\put(79,22){\red{\vector(-2.3,1){15}}}
\put(20,14){\red{\vector(0,1){13}}}
\put(61,26){\red{\vector(0,-1){13}}}
\put(64,12){\red{\vector(2.3,1){15}}}

\end{picture}
\qquad
\begin{picture}(100,40)(0,0)
\thicklines
\put(10,0){\circle*{5}}
\put(30,0){\circle{5}}
\put(70,0){\circle*{5}}
\put(90,0){\circle{5}}
\put(10,20){\circle{5}}
\put(30,20){\circle*{5}}
\put(50,20){\circle{5}}
\put(70,20){\circle*{5}}
\put(10,40){\circle*{5}}
\put(50,40){\circle*{5}}
\put(70,40){\circle{5}}
\put(90,40){\circle{5}}

\put(0,0){\line(1,0){27.5}}
\put(32.5,0){\line(1,0){55}}
\put(92.5,0){\line(1,0){7.5}}
\put(12.5,20){\line(1,0){35}}
\put(52.5,20){\line(1,0){15}}
\put(0,40){\line(1,0){47.5}}
\put(52.5,40){\line(1,0){15}}
\put(72.5,40){\line(1,0){15}}
\put(92.5,40){\line(1,0){7.5}}

\put(10,2.5){\line(0,1){15}}
\put(10,22.5){\line(0,1){15}}
\put(30,17.5){\line(0,-1){15}}
\put(50,22.5){\line(0,1){15}}
\put(70,2.5){\line(0,1){15}}
\put(70,22.5){\line(0,1){15}}
\put(90,2.5){\line(0,1){35}}

\put(20,10){\red{\circle*{5}}}
\put(61,10){\red{\circle*{5}}}

\put(20,30){\red{\circle*{5}}}
\put(61,30){\red{\circle*{5}}}
\put(83,20){\red{\circle*{5}}}

\put(56,30){\red{\vector(-1,0){33}}}
\put(56,10){\red{\vector(-1,0){33}}}
\put(20,14){\red{\vector(0,1){13}}}
\put(61,14){\red{\vector(0,1){13}}}
\put(64,28){\red{\vector(2.3,-1){16}}}
\put(24,28){\red{\vector(2,-1){33}}}

\end{picture}
\qquad
\begin{picture}(100,40)(0,0)
\thicklines
\put(10,0){\circle*{5}}
\put(30,0){\circle{5}}
\put(70,0){\circle*{5}}
\put(90,0){\circle{5}}
\put(10,20){\circle{5}}
\put(30,20){\circle*{5}}
\put(50,20){\circle{5}}
\put(50,0){\circle*{5}}
\put(10,40){\circle*{5}}
\put(50,40){\circle*{5}}
\put(70,40){\circle{5}}
\put(90,40){\circle{5}}

\put(0,0){\line(1,0){27.5}}
\put(32.5,0){\line(1,0){55}}
\put(92.5,0){\line(1,0){7.5}}
\put(12.5,20){\line(1,0){35}}
\
\put(0,40){\line(1,0){47.5}}
\put(52.5,40){\line(1,0){15}}
\put(72.5,40){\line(1,0){15}}
\put(92.5,40){\line(1,0){7.5}}

\put(10,2.5){\line(0,1){15}}
\put(10,22.5){\line(0,1){15}}
\put(30,17.5){\line(0,-1){15}}
\put(50,22.5){\line(0,1){15}}
\put(50,2.5){\line(0,1){15}}
\put(70,2.5){\line(0,1){35}}
\put(90,2.5){\line(0,1){35}}

\put(20,10){\red{\circle*{5}}}
\put(40,10){\red{\circle*{5}}}

\put(30,30){\red{\circle*{5}}}
\put(60,20){\red{\circle*{5}}}
\put(83,20){\red{\circle*{5}}}

\put(56,22){\red{\vector(-3,1){22}}}
\put(36,10){\red{\vector(-1,0){13}}}
\put(22,14){\red{\vector(1,2){6}}}
\put(44,12){\red{\vector(2,1){13}}}

\put(64,20){\red{\vector(1,0){15}}}
\put(32,26){\red{\vector(1,-2){6}}}

\end{picture}
\vspace{.05in}
\end{center}
\caption{The first two plabic graphs are related by a square move; their quivers are obtained from each other by a single mutation.
The second and third plabic graphs are related by a flip move and have isomorphic quivers.
}
\vspace{-.25in}
\label{fig:plabic-moves-are-mutations}
\end{figure}
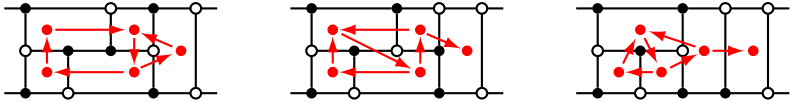

\begin{remark}
\label{rem:q-not=>p}
The converse to Proposition~\ref{pr:plabic-vs-quivers} is false, see \cite[Remark~6.10]{FPST}:
there exist plabic graphs which are not move equivalent even though 
their quivers are isomorphic (hence mutation equivalent). 
\end{remark}

It is natural to wonder whether the class of quivers (and by extension, cluster algebras)
arising from plabic graphs is substantially narrower than the class of all quivers. 
The following  result shows that in some sense, no generality is lost 
by restricting the theory of quiver mutations to plabic graphs:

\begin{theorem}
\label{th:plabic-universal}
For any integer $n$, there exists a plabic graph with $O(n^4)$ edges and $O(n^4)$ faces 
whose associated quiver is $n$-universal. 
In particular, any quiver is a full subquiver of a quiver
mutation equivalent to a quiver of a plabic graph. 
\end{theorem}

Theorem~\ref{th:plabic-universal}
can be regarded as a strengthening of Theorem~\ref{thm:planar-universal} 
since the quiver of any plabic graph is planar.

The proof of Theorem~\ref{th:plabic-universal} will rely on a couple of lemmas.

\begin{lemma}
\label{lem:quivers-from-plabic}
Let $Q$ be a planar quiver properly embedded in the plane. 
Assume that 
\begin{itemize}[leftmargin=.3in]
\item[\rm{(a)}]
$Q$ is connected; 
\item[\rm{(b)}]
$Q$ has no univalent vertices; 
\item[\rm{(c)}]
the closure of each bounded face of~$Q$ is simply connected; and 
\item[\rm{(d)}]
the boundary of each bounded face of~$Q$ is an oriented cycle.
\end{itemize}
Then 
there exists a plabic graph $P$ such that $Q(P)$ is isomorphic to~$Q$. 
\end{lemma}

\begin{proof}
Let $U$ denote the unique unbounded face of~$Q$. 
We construct the plabic~graph $P$ as follows. 
(Consult Figure~\ref{fig:quiver-to-plabic}.) 
We begin by placing one black (resp., white) vertex of~$P$ inside each bounded face of~$Q$ 
whose boundary is oriented clockwise (resp., counter-clockwise). 
For each arrow~$a$ of~$Q$ that lies on the boundary of the unbounded face~$U$,
let us place a vertex of~$P$ inside~$U$ next to~$a$. 
(There will be many vertices of~$P$ inside~$U$, one for each boundary segment of~$U$.) 
Color that vertex black (resp., white) if $U$ lies on the right (resp., left) of~$a$ as we move along~$a$. 
If $a$ has $U$ on both sides of it, place a black vertex on its right and a white vertex on its~left. 

We now describe the edges of~$P$. 
For each arrow~$a$ of~$Q$ separating two bounded (necessarily distinct) faces of~$Q$, 
connect the vertices of~$P$ placed in those faces with an edge of~$P$ going across~$a$.
If one or both faces bordering an arrow~$a$ are unbounded, then use instead 
the vertices of~$P$ located near~$a$. 
We then walk around the boundary of the unbounded face~$U$ 
and cyclically connect all the vertices of~$P$ located~there. 
Finally, we make our plabic graph trivalent by ``uncontracting'' each vertex of degree $\ge 4$ into a trivalent tree, 
see the paragraph preceding Corollary~12.4 in~\cite{postnikov}; cf.\ also Figure~\ref{fig:quiver-to-plabic}.
It is straightforward to verify that the quiver associated with the resulting plabic graph~$P$
is the original quiver~$Q$. 
\end{proof}

\begin{figure}[ht]
\begin{center}
\setlength{\unitlength}{2pt} 
\begin{picture}(70,50)(5,5)
\thicklines 

\multiput(5,40)(40,0){2}{\circle{2}}
\multiput(20,55)(40,0){2}{\circle{2}}
\multiput(20,40)(40,0){2}{\circle*{2}}
\multiput(35,40)(40,0){2}{\circle{2}}
\put(40,35){\circle{2}}
\put(20,25){\circle{2}}
\put(40,25){\circle*{2}}
\put(45,20){\circle*{2}}
\put(60,20){\circle{2}}
\put(75,20){\circle*{2}}
\put(60,5){\circle*{2}}

\multiput(10,30)(20,0){2}{\red{\circle*{2}}}
\multiput(50,30)(20,0){2}{\red{\circle*{2}}}
\multiput(10,50)(20,0){2}{\red{\circle*{2}}}
\multiput(50,50)(20,0){2}{\red{\circle*{2}}}
\multiput(50,10)(20,0){2}{\red{\circle*{2}}}

\multiput(6,40)(40,0){2}{\line(1,0){13}}
\multiput(21,40)(40,0){2}{\line(1,0){13}}
\multiput(20,54)(0,-15){2}{\line(0,-1){13}}
\put(60,39){\line(0,-1){18}}
\multiput(60,54)(0,-35){2}{\line(0,-1){13}}
\put(75,39){\line(0,-1){18}}
\put(46,20){\line(1,0){13}}
\put(61,20){\line(1,0){13}}
\put(40,26){\line(0,1){8}}
\put(21,25){\line(1,0){18}}
\qbezier(5,50)(5,55)(10,55)
\put(5,41){\line(0,1){9}}
\put(10,55){\line(1,0){9}}
\qbezier(5,30)(5,25)(10,25)
\put(5,39){\line(0,-1){9}}
\put(10,25){\line(1,0){9}}
\qbezier(35,50)(35,55)(30,55)
\put(35,41){\line(0,1){9}}
\put(30,55){\line(-1,0){9}}
\qbezier(75,50)(75,55)(70,55)
\put(75,41){\line(0,1){9}}
\put(70,55){\line(-1,0){9}}
\qbezier(45,50)(45,55)(50,55)
\put(45,41){\line(0,1){9}}
\put(50,55){\line(1,0){9}}
\qbezier(45,10)(45,5)(50,5)
\put(45,19){\line(0,-1){9}}
\put(50,5){\line(1,0){9}}
\qbezier(75,10)(75,5)(70,5)
\put(75,19){\line(0,-1){9}}
\put(70,5){\line(-1,0){9}}
\qbezier(35,39)(35,35)(39,35)
\qbezier(45,39)(45,35)(41,35)
\qbezier(45,21)(45,25)(41,25)

\red{
\multiput(12,50)(40,0){2}{\vector(1,0){16}}
\put(32,30){\vector(1,0){16}}
\put(52,10){\vector(1,0){16}}
\put(28,30){\vector(-1,0){16}}
\put(68,30){\vector(-1,0){16}}
\put(50,28){\vector(0,-1){16}}
\put(70,12){\vector(0,1){16}}
\multiput(10,32)(40,0){2}{\vector(0,1){16}}
\multiput(30,48)(40,0){2}{\vector(0,-1){16}}
}
\end{picture}
\qquad\qquad
\begin{picture}(70,50)(5,5)
\thicklines 

\multiput(5,40)(40,0){2}{\circle{2}}
\multiput(20,55)(40,0){2}{\circle{2}}
  \multiput(18,42)(40,0){2}{\circle*{2}}
  \multiput(22,38)(40,0){2}{\circle*{2}}
\multiput(35,40)(40,0){2}{\circle{2}}
\put(40,35){\circle{2}}
\put(20,25){\circle{2}}
\put(40,25){\circle*{2}}
\put(45,20){\circle*{2}}
\put(58,22){\circle{2}}
\put(62,18){\circle{2}}
\put(75,20){\circle*{2}}
\put(60,5){\circle*{2}}

\multiput(10,30)(20,0){2}{\red{\circle*{2}}}
\multiput(50,30)(20,0){2}{\red{\circle*{2}}}
\multiput(10,50)(20,0){2}{\red{\circle*{2}}}
\multiput(50,50)(20,0){2}{\red{\circle*{2}}}
\multiput(50,10)(20,0){2}{\red{\circle*{2}}}

\qbezier(6,40)(16,40)(18,42)
\qbezier(20,54)(20,44)(18,42)
\qbezier(34,40)(24,40)(22,38)
\qbezier(20,26)(20,36)(22,38)
\put(18,42){\line(1,-1){4}}

\qbezier(46,40)(56,40)(58,42)
\qbezier(60,54)(60,44)(58,42)
\qbezier(74,40)(64,40)(62,38)
\qbezier(60,30)(60,36)(62,38)
\put(58,42){\line(1,-1){4}}

\qbezier(46,20)(56,20)(57.25,21.25)
\qbezier(60,30)(60,24)(58.75,22.75)
\qbezier(74,20)(64,20)(62.75,18.75)
\qbezier(60,6)(60,16)(61.25,17.25)
\put(58.75,21.25){\line(1,-1){2.5}}

\put(75,39){\line(0,-1){18}}
\put(40,26){\line(0,1){8}}
\put(21,25){\line(1,0){18}}
\qbezier(5,50)(5,55)(10,55)
\put(5,41){\line(0,1){9}}
\put(10,55){\line(1,0){9}}
\qbezier(5,30)(5,25)(10,25)
\put(5,39){\line(0,-1){9}}
\put(10,25){\line(1,0){9}}
\qbezier(35,50)(35,55)(30,55)
\put(35,41){\line(0,1){9}}
\put(30,55){\line(-1,0){9}}
\qbezier(75,50)(75,55)(70,55)
\put(75,41){\line(0,1){9}}
\put(70,55){\line(-1,0){9}}
\qbezier(45,50)(45,55)(50,55)
\put(45,41){\line(0,1){9}}
\put(50,55){\line(1,0){9}}
\qbezier(45,10)(45,5)(50,5)
\put(45,19){\line(0,-1){9}}
\put(50,5){\line(1,0){9}}
\qbezier(75,10)(75,5)(70,5)
\put(75,19){\line(0,-1){9}}
\put(70,5){\line(-1,0){9}}
\qbezier(35,39)(35,35)(39,35)
\qbezier(45,39)(45,35)(41,35)
\qbezier(45,21)(45,25)(41,25)

\red{
\multiput(12,50)(40,0){2}{\vector(1,0){16}}
\put(32,30){\vector(1,0){16}}
\put(52,10){\vector(1,0){16}}
\put(28,30){\vector(-1,0){16}}
\put(68,30){\vector(-1,0){16}}
\put(50,28){\vector(0,-1){16}}
\put(70,12){\vector(0,1){16}}
\multiput(10,32)(40,0){2}{\vector(0,1){16}}
\multiput(30,48)(40,0){2}{\vector(0,-1){16}}
}
\end{picture}
\end{center}
\caption{Constructing a plabic graph from a planar quiver 
satisfying the conditions in Lemma~\ref{lem:quivers-from-plabic}.
The graph on the right is obtained from the one on the left by ``uncontracting'' three vertices of degree~4.}
\vspace{-.25in}
\label{fig:quiver-to-plabic}
\end{figure}
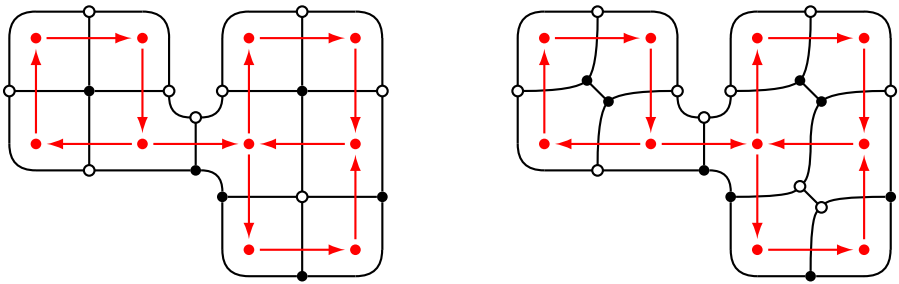

\begin{lemma}
\label{lem:oriented-faces}
Any planar quiver can be embedded (as a full subquiver) into a planar quiver satisfying 
conditions {\rm (a)--(d)} in Lemma~\ref{lem:quivers-from-plabic}. 
\end{lemma}

\begin{proof}
We first add pairs of arrows $\,\red{\bullet\!\to\!\bullet\!\to\!\bullet}\,$ to the quiver 
(here the vertex in the middle is new)
to make it connected, to get rid of univalent vertices,
and to cut up each bounded face that is not simply connected
into pieces that are. 
(To convince yourself that the latter can always be accomplished, 
recall that a quiver cannot contain loops, i.e., arrows connecting a vertex to itself.) 
 
It remains to take care of condition~(d). Let $F$ be a bounded face of the resulting quiver.
Suppose that $F$ is not oriented. 
Let $Q_F$ denote the quiver formed by the vertices and arrows of~$Q$
lying on the boundary of~$F$. 
(There might exist arrows which connect some of those vertices to each other
but do not lie on~$\partial F$;
those arrows are not included in~$Q_F$.) 
We now augment $Q$ by adding an extra vertex~$v_F$ lying inside~$F$, 
and by adding new arrows to~$Q$ as follows:  
\begin{itemize}[leftmargin=.2in]
\item for every sink~$u$ of~$Q_F$, draw an arrow $u\color{cyan}\to\color{black} v_F$; 
\item for every source~$u$ of~$Q_F$, draw an arrow $v_F\color{cyan}\to\color{black} u$. 
\end{itemize}
Once this is done for  all such faces~$F$, we obtain a quiver all of whose bounded faces are oriented. 
See Figure~\ref{fig:making-faces-oriented}. 
\end{proof}

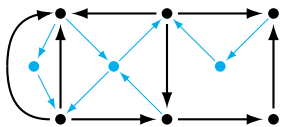
\begin{figure}[ht]
\begin{center}
\setlength{\unitlength}{2pt} 
\begin{picture}(60,20)(-8,0)
\thicklines 

\multiput(0,0)(20,0){3}{\circle*{2}}
\multiput(0,20)(20,0){3}{\circle*{2}}
\put(2,0){\vector(1,0){16}}
\put(22,0){\vector(1,0){16}}
\put(22,20){\vector(1,0){16}}
\put(18,20){\vector(-1,0){16}}
\put(0,2){\vector(0,1){16}}
\put(20,18){\vector(0,-1){16}}
\put(40,2){\vector(0,1){16}}
\qbezier(-2,0)(-10,0)(-10,10)
\qbezier(-2,20)(-10,20)(-10,10)
\put(-2.3,19.95){\vector(6,1){1}}
\thinlines
\color{cyan}
\multiput(10,10)(20,0){2}{\circle*{2}}
\put(-5,10){\circle*{2}}
\put(-1,18){\vector(-1,-2){3.3}}
\put(-4.2,8.4){\vector(1,-2){3.3}}
\put(11,11){\vector(1,1){8}}
\put(9,9){\vector(-1,-1){8}}
\put(1,19){\vector(1,-1){8}}
\put(19,1){\vector(-1,1){8}}
\put(29,11){\vector(-1,1){8}}
\put(39,19){\vector(-1,-1){8}}

\color{black}
\end{picture}
\vspace{-.1in}
\end{center}
\caption{Embedding a planar quiver into a quiver with oriented bounded~faces.}
\vspace{-.1in}
\label{fig:making-faces-oriented}
\end{figure}

\pagebreak[3]

\begin{proof}[Proof of Theorem~\ref{th:plabic-universal}] 

Lemmas~\ref{lem:quivers-from-plabic} and~\ref{lem:oriented-faces} imply that 
any planar quiver~$Q$ can be embedded (as a full subquiver) 
into a quiver $Q(P)$ coming from a plabic graph~$P$. 
This embedding is illustrated in Figure~\ref{fig:extended-Somos-into-plabic} 
for the case of the extended Somos quiver. 

Taking account of the number of vertices and arrows added in the course of such embedding
(cf.\ the proofs of Lemmas~\ref{lem:quivers-from-plabic} and~\ref{lem:oriented-faces}),
we conclude that if a planar quiver~$Q$ has $k$ vertices and $\ell$ arrows,
then the number of vertices (resp., arrows) of the ambient quiver $Q(P)$ are both bounded by $O(k+\ell)$. 
Letting $Q$ be an $n$-universal planar quiver 
from Theorem~\ref{thm:planar-universal}, we obtain the statement of Theorem~\ref{th:plabic-universal}. 
\end{proof}

\newsavebox{\wb}
\setlength{\unitlength}{2pt} 
\savebox{\wb}(10,5)[bl]{
\thicklines 
\put(0,0){\circle{2}} 
\put(10,0){\circle*{2}} 
\qbezier(0.8,0.6)(5,3)(9.2,0.6)
\qbezier(0.8,-0.6)(5,-3)(9.2,-0.6)
}

\newsavebox{\bw}
\setlength{\unitlength}{2pt} 
\savebox{\bw}(10,5)[bl]{
\thicklines 
\put(0,0){\circle*{2}} 
\put(10,0){\circle{2}} 
\qbezier(0.8,0.6)(5,3)(9.2,0.6)
\qbezier(0.8,-0.6)(5,-3)(9.2,-0.6)
}

\newsavebox{\horiz}
\setlength{\unitlength}{2pt} 
\savebox{\horiz}(10,0)[bl]{
\thicklines 
\put(1,0){\line(1,0){8}}
}

\begin{figure}[ht]
\begin{center}
\vspace{-.4in}
\begin{picture}(130,150)(12,-15) 
\thicklines
\put(20,60){{\usebox{\wb}}} 
\put(40,60){{\usebox{\wb}}} 
\multiput(70,40)(0,40){2}{{\usebox{\bw}}} 
\multiput(100,40)(0,40){2}{{\usebox{\bw}}} 
\put(9.5,61){\line(-1,2){9}}
\put(9.5,59){\line(-1,-2){9}}
\put(60.5,61){\line(1,2){9}}
\put(60.5,59){\line(1,-2){9}}
\put(90,40){\line(0,1){39}}
\put(120,40){\line(0,1){39}}
\put(120.5,81){\line(1,2){9}}
\put(120.5,39){\line(1,-2){9}}
\multiput(10,60)(20,0){3}{{\usebox{\horiz}}} 
\multiput(80,40)(0,40){2}{{\usebox{\horiz}}} 
\multiput(90,40)(0,40){2}{{\usebox{\horiz}}} 
\multiput(110,40)(0,40){2}{{\usebox{\horiz}}} 
\put(0,40){\circle{2}} 
\put(0,80){\circle{2}} 
\put(10,60){\circle*{2}} 
\put(60,60){\circle{2}} 
\put(90,40){\circle*{2}} 
\put(90,80){\circle{2}} 
\put(120,40){\circle*{2}} 
\put(120,80){\circle*{2}} 
\put(130,20){\circle{2}} 
\put(130,100){\circle{2}} 
\qbezier(0.6,39.2)(60,-30)(129.2,19.4)
\qbezier(-0.6,40.8)(-14,60)(-0.6,79.2)
\qbezier(0.6,80.8)(60,150)(129.2,100.6)
\qbezier(130.8,20.6)(178,60)(130.8,99.4)

\put(0,60){\red{\makebox(0,0){$u$}}}
\put(0,60){\red{\circle{6}}}
\put(80,10){\red{\makebox(0,0){$2$}}}
\put(80,10){\red{\circle{6}}}
\put(80,60){\red{\makebox(0,0){$4$}}}
\put(80,60){\red{\circle{6}}}
\put(80,110){\red{\makebox(0,0){$3$}}}
\put(80,110){\red{\circle{6}}}
\put(110,60){\red{\makebox(0,0){$1$}}}
\put(110,60){\red{\circle{6}}}
\put(135,60){\red{\makebox(0,0){$v$}}}
\put(135,60){\red{\circle{6}}}

\red{
\qbezier(2,56)(30,8)(76,8)
\qbezier(2,64)(30,112)(76,112)
\put(2.6,55){\vector(-1,1.7){1}}
\put(76,112){\vector(1,0){1}}

\qbezier(75,10)(-45,60)(75,110)
\put(75,10){\vector(2.5,-1){1}}

\qbezier(75,12)(-5,60)(75,108)
\put(75,12){\vector(1.7,-1){1}}

\qbezier(76.5,13.5)(30,60)(76.5,106.5)
\put(76,14){\vector(1,-1){1}}

\qbezier(79,14)(50,50)(76,58)
\put(76,58){\vector(2.5,1){1}}

\qbezier(79,106)(50,70)(76,62)
\put(78.5,105.4){\vector(1,1.2){1}}

\qbezier(82,14)(88,40)(82,56)
\put(82.7,53.9){\vector(-1,3){1}}

\qbezier(82,106)(90,90)(107,63)
\put(82.6,104.8){\vector(-1,2){1}}

\put(84,60){\vector(1,0){22}}

\qbezier(84,11)(125,30)(112,56)
\put(112.5,55){\vector(-1,2){1}}

\qbezier(84,109)(125,90)(112,64)
\put(84.4,108.8){\vector(-2,1){1}}

\qbezier(84,9)(135,10)(135,56)
\put(84.3,9){\vector(-1,0){1}}

\qbezier(84,111)(135,110)(135,64)
\put(135,64.3){\vector(0,-1){1}}
}
\end{picture} 
\vspace{-.4in}
\end{center}
\caption{Embedding the extended Somos quiver~$Q$ (see Figure~\ref{fig:flattened-Somos-4})
into a quiver of a plabic graph~$P$.
Removing the vertices of the latter quiver 
corresponding to the bigon faces of~$P$, we obtain~$Q$ (shown in the picture). 
}
\vspace{-.1in}
\label{fig:extended-Somos-into-plabic}
\end{figure}
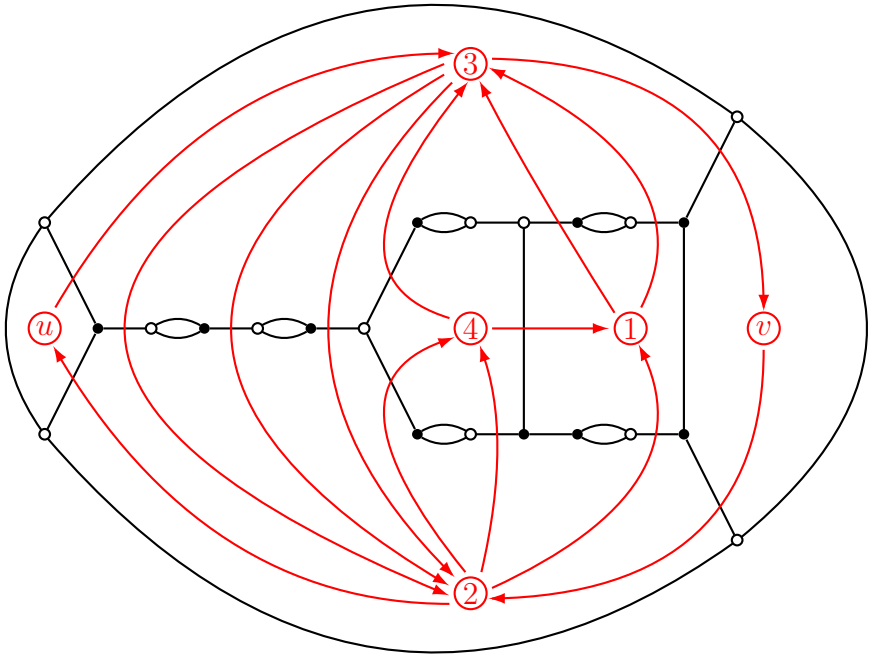

\begin{remark}
\label{rem:reduced-plabic-not-universal}
There is a large body of research focusing on the important subclass of 
\emph{reduced} plabic graphs. Reduced plabic graphs describe cluster structures in 
Grassmannians, basic affine spaces, and most generally, arbitrary positroid varieties~\cite{Galashin-Lam}. \linebreak[3]
Remarkably, the analogue of Theorem~\ref{th:plabic-universal} fails for reduced plabic graphs,
for the following reason. 
As shown by N.~Ford and K.~Serhiyenko \cite[Theorem~1.2]{ford-serhiyenko},
the quiver associated with any reduced plabic graph has a \emph{reddening sequence}. 
This property is preserved by quiver mutations and by passing to a full subquiver, 
see G.~Muller \cite[Theorem 17, Corollary~19]{muller-max-green}. 
Consequently any quiver without a reddening sequence
(e.g., the Markov quiver, cf.\ Figure~\ref{fig:Markov-quiver})  
cannot be embedded into a quiver of a reduced plabic graph. 
Thus, restricting the study of plabic graphs to the reduced \linebreak[3]
case limits the theory
to a proper subclass of ``nice'' quivers (and associated cluster structures)
whereas general plabic graphs contain arbitrarily ``nasty'' cluster types. 
\end{remark}

\newpage

\section{Universal skew-symmetrizable matrices}
\label{sec:universal-matrices}

In this section, we extend Theorem~\ref{thm:main-thm-quivers} to the wider generality of skew-sym\-met\-rizable matrices. 
We start by recalling the basic background, cf., e.g., \cite[Section~2.7]{FWZ}. 

\begin{definition} 
An $n \times n$ matrix $B = (b_{ij})$ with integer entries is called
\emph{skew-symmetrizable} if there exist positive integers $d_1,..., d_n$ such that $d_ib_{ij} = -d_j b_{ji}$ for all $i,j\in\{1,\dots,n\}$. 
Equivalently, there exists a diagonal matrix $D=\operatorname{diag}(d_1,\dots,d_n)$ with positive diagonal entries $d_1,...,d_n$ such that the matrix $DB$ is skew-symmetric. Such a matrix $D$ is called a \emph{symmetrizer} for~$B$. 
\end{definition}

\begin{definition} 
\label{def:matrix-mutation}
Let $B$ be an $n \times n$  skew-symmetrizable
integer matrix. 
For $k\in\{1,...,n\}$, the \emph{matrix mutation $\mu_k$ in direction $k$} transforms $B$ into
the  $n \times n$ matrix $\mu_k(B)=(b_{ij}')$ whose entries are given as follows:
\begin{equation}
\label{eq:matrix-mutation}
b_{ij}'=
\begin{cases} -b_{ij} &\text{ if }i=k\text{ or }j=k; \\
b_{ij}+b_{ik}b_{kj} &\text{ if }b_{ik}>0\text{ and }b_{kj}>0;\\
b_{ij}-b_{ik}b_{kj} &\text{ if }b_{ik}<0\text{ and }b_{kj}<0;\\ 
b_{ij} &\text{ otherwise.}  
\end{cases}
\end{equation}
Two skew-symmetrizable matrices $B$ and $B'$ are called \emph{mutation
equivalent} (denoted \hbox{$B\sim B'$}) if one can get from $B$ to $B'$ by a sequence of mutations.
\end{definition}

\begin{remark}
To any quiver~$Q$ with vertices $1,\dots,n$, we can associate a skew-symmetric matrix $B=B(Q)=(b_{ij})$ by setting
\[
b_{ij}=\begin{cases}
\quad \, \text{number of arrows pointing from $i$ to~$j$} &\text{if there are any;}\\
-(\text{number of arrows pointing from $j$ to~$i$}) &\text{otherwise.}
\end{cases}
\]
Under this correspondence, mutations of quivers translate into mutations of associated matrices. Thus, Definition~\ref{def:matrix-mutation} is a generalization of Definition~\ref{def:quiver-mutation}. 
\end{remark}

It is easy to see that if $D$ is a symmetrizer for~$B$, and $B\sim B'$, then $D$ is also a symmetrizer for~$B'$. 

\begin{definition}
Let $B$ be a skew-symmetrizable $n\times n$ matrix, and $I$ a subset of $\{1,\dots,n\}$. We denote by $B_I$ the \emph{principal submatrix} of~$B$ supported on~$I$, i.e., the matrix obtained from~$B$ by selecting the rows and columns belonging to~$I$. 
\end{definition}

\begin{remark}
The notion of a principal submatrix is a straightworward extension of the notion of a full subquiver, see Definition~\ref{def:full-subquiver}: for a matrix $B(Q)$ associated with a quiver~$Q$, we have $B(Q)_I=B(Q_I)$. 

Lemma~\ref{lem:restriction-mutation} extends to this generality: for $k\in I$, we have $\mu_k(B_I)=(\mu_k(B))_I$. 
\end{remark}

\pagebreak[3]

We can now extend Definition~\ref{def:n-universal}. 

\begin{definition}
Let $D=\operatorname{diag}(d_1,\dots,d_n)$ be a diagonal matrix with positive integer diagonal entries $d_1,...,d_n$. 
A skew-symmetrizable matrix $B$ is called $D$-\emph{universal} if every $n\times n$ skew-symmetrizable matrix with symmetrizer $D$  is a principal submatrix of a matrix mutation equivalent to  $B$. 
\end{definition}

Our main result (Theorem~\ref{thm:main-thm-quivers}) directly extends to the skew-symmetrizable case. 

\begin{theorem}
\label{main_thm1}
For any $n\times n$ diagonal matrix $D$ with positive integer entries, there exists a $D$-universal matrix of size $(2n^2\!-\!n)\!\times\!(2n^2\!-\!n)$ with $2(7n^2\!-\!7n)$ nonzero~entries. 
\end{theorem}

To prove Theorem~\ref{main_thm1}, we will need a couple of lemmas. 

\begin{lemma}
\label{lem:mu-BD}
Let $B$ be an $n\times n$ skew-symmetric matrix, and $H=\operatorname{diag}(h_1,\dots,h_n)$ a diagonal matrix with positive integer entries. If $h_k=1$, then $\mu_k(BH)=\mu_k(B)H$. 
\end{lemma}

\begin{proof}
Let $B=(b_{ij})$. Then we have, for any $i,j\in\{1,\dots,n\}$ (cf.~\eqref{eq:matrix-mutation}): 
\[
(\mu_k(BH))_{ij}=\begin{cases} -b_{ij}h_j &\text{ if }i=k\text{ or }j=k; \\
b_{ij}h_j+b_{ik}h_k b_{kj} h_j &\text{ if }b_{ik}>0\text{ and }b_{kj}>0;\\
b_{ij}h_j-b_{ik}h_k b_{kj} h_j&\text{ if }b_{ik}<0\text{ and }b_{kj}<0;\\ 
b_{ij}h_j &\text{ otherwise.}  
\end{cases}
\] 
Since $h_k=1$, it follows that $(\mu_k(BH))_{ij}=(\mu_k(B))_{ij}h_j=(\mu_k(B)H)_{ij}$. 
\end{proof}

\begin{lemma}
\label{lem:bij-divisible}
Let $B=(b_{ij})$ be a skew-symmetrizable $n\times n$ matrix with symmetrizer $D=\operatorname{diag}(d_1,\dots,d_n)$. 
For $i,j\in\{1,\dots,n\}$, denote
\begin{equation}
\label{eq:hij}
h_{ij}=\frac{d_j}{\gcd(d_i,d_j)} .
\end{equation}
Then 
there exists an integer $k$ such that 
\[
\begin{bmatrix}
0 & b_{ij} \\
b_{ji} & 0
\end{bmatrix}
= \begin{bmatrix}
0 & kh_{ij} \\
-kh_{ji} & 0
\end{bmatrix} .
\]
\end{lemma}

\begin{proof}
By the definition of a symmetrizer, we have $d_i b_{ij}=-d_j b_{ji}$. Dividing by $\gcd(d_i,d_j)$, we get $h_{ji} b_{ij}=-h_{ij} b_{ji}$.
It remains to observe that $\gcd(h_{ij},h_{ji})=1$, and therefore the number 
$k=\frac{b_{ij}}{h_{ij}}=-\frac{b_{ji}}{h_{ji}}$ 
 is an integer. 
\end{proof}

\begin{proof}[Proof of Theorem~\ref{main_thm1}]
As explained below, the requisite $D$-universal matrix is constructed in the same way as the $n$-universal quiver~$\overline Q$ in Theorem~\ref{thm:main-thm-quivers}, with slight modifications. 

We continue to use the notation from the proof of Theorem~\ref{thm:main-thm-quivers}. 
Let $\overline B=B(\overline Q)$ be the skew-symmetric matrix corresponding to the quiver~$\overline Q$. All nonzero entries of~$\overline B$ are located inside $\binom{n}{2}$ blocks $\overline B_{ij}$ of size $6\times 6$, labeled by the pairs $i,j\in\{1,\dots,n\}$ of vertices in~$Q_\bullet$. Each block $\overline B_{ij}$ has the same form, and corresponds to a copy (here denoted~$Q_{ij}$) of the extended Somos-4 quiver. The marked vertices of~$Q_{ij}$ (those labeled $u$ and~$v$ in Figures~\ref{fig:extended-somos-4} and~\ref{fig:flattened-Somos-4}) are the vertices $i$ and~$j$ of~$Q_\bullet$. We label the remaining four vertices by the triples $(i,j,1)$, $(i,j,2)$, $(i,j,3)$, $(i,j,4)$. With this notation, the block $\overline B_{ij}$ is represented as follows:
\begin{equation}
\label{eq:overline-Bij}
\begin{array}{c||cc|cccc}
& i & j & (i,j,1) & (i,j,2) & (i,j,3) & (i,j,4) \\[1pt]
\hline
\hline
\\[-12pt]
i & 0 & 0 & 0 & -1 & 1 & 0  \\
j & 0 & 0 & 0 & 1 & -1 & 0 \\
\hline
\\[-11pt]
(i,j,1) & 0 & 0  & 0 & -1 & 2 & -1  \\
(i,j,2) & 1 & -1 & 1 & 0 & -3 & 2   \\
(i,j,3) & -1 &1 &-2  & 3 & 0 & -1 \\
(i,j,4) & 0 & 0 & 1  & -2 & 1 &0 
\end{array}
\end{equation}

Let $D=\operatorname{diag}(d_1,\dots,d_n)$. For $i,j\in\{1,\dots,n\}$, we denote
\[
H_{ij}=\operatorname{diag}(h_{ji},h_{ij},1,1,1,1), 
\]
where we use the notation~\eqref{eq:hij}. 

To construct the desired $D$-universal matrix~$\overline B_D$, we replace each block $\overline B_{ij}$ in~$\overline B$ by the product~$\overline B_{ij} H_{ij}$. In other words, we replace~\eqref{eq:overline-Bij} by the $6\times 6$ block
\begin{equation}
\label{eq:overline-BijH}
\begin{array}{c||cc|cccc}
& i & j & (i,j,1) & (i,j,2) & (i,j,3) & (i,j,4) \\[1pt]
\hline
\hline
\\[-12pt]
i & 0 & 0 & 0 & -1 & 1 & 0  \\
j & 0 & 0 & 0 & 1 & -1 & 0 \\
\hline
\\[-11pt]
(i,j,1) & 0 & 0  & 0 & -1 & 2 & -1  \\
(i,j,2) & h_{ji} & -h_{ij} & 1 & 0 & -3 & 2   \\
(i,j,3) & -h_{ji} &h_{ij} &-2  & 3 & 0 & -1 \\
(i,j,4) & 0 & 0 & 1  & -2 & 1 &0 
\end{array}
\end{equation}
The resulting $(2n^2-n)\times(2n^2-n)$ matrix~$\overline B_D$ is skew-symmetrizable: its symmetrizer has diagonal entry $d_i$ corresponding to each vertex $i$ in~$Q_\bullet$, and diagonal entry~$\gcd(d_i,d_j)$ for each vertex labeled $(i,j,1)$, $(i,j,2)$, $(i,j,3)$, or $(i,j,4)$.  

It remains to show that the matrix $\overline B_D$ is $D$-universal. The general shape of the argument is the same as in the quiver case: we mutate each block separately; these~mutations do not interfere with each other; and at the end, we restrict to the principal submatrix supported on the indexing set $\{1,\dots,n\}$. 
Since the mutated matrix has the same symmetrizer as~$\overline B_D$, the restricted matrix will have symmetrizer~$D$. 

We now need to prove that for any skew-symmetrizable matrix $B=(b_{ij})$ with symmetrizer~$D$, and any $i,j\in\{1,\dots,n\}$, there exists a sequence of mutations, using only the indices $(i,j,1)$, $(i,j,2)$, $(i,j,3)$, or~$(i,j,4)$, which transform the $6\times 6$ matrix $\overline B_{ij} H_{ij}$ shown in~\eqref{eq:overline-BijH} into a matrix whose upper-left $2\times 2$ block is $\Bigl[\begin{array}{cc} 0 & b_{ij}\\ b_{ji} & 0 \end{array}\Bigr]$. 
By Lemma~\ref{lem:bij-divisible}, this $2\times 2$ block has the form $ \Bigl[\begin{array}{cc}
0 & kh_{ij} \\
-kh_{ji} & 0
\end{array}\Bigr]$, for some $k\in\mathbb{Z}$. 
Let us take a sequence of mutations whose indices replicate the mutation sequence required to obtain the desired outcome in the quiver case, i.e., 
to transform $\overline B_{ij}$ into a matrix with the upper-left block $\Bigl[\begin{array}{cc}
0 & k \\
-k & 0
\end{array}\Bigr]$. 
By Lemma~\ref{lem:mu-BD} (which applies since the corresponding components of $H_{ij}$ are equal to~1), each of these mutations commutes with multiplication by $H_{ij}$ on the right, and we are done. 
\end{proof}

\newpage

\section{Hereditary properties and universal collections}
\label{sec:hereditary+universal}

We begin by reviewing the concept of hereditary properties, cf.\ \cite[Section~4.1]{FWZ}. We will then explain the connection between this concept and the notion of universality. 

\begin{definition}[{\cite[Definition 4.1.3]{FWZ}}]
\label{def:hereditary}
A property of quivers (or more generally skew-symmetrizable matrices) is called \emph{hereditary} if it is preserved under restriction to a full subquiver (resp., a principal submatrix). 

Of particular importance are the properties of quivers (resp., skew-symmetrizable matrices) which are both hereditary and \emph{mutation-invariant}, i.e., preserved under mutations. 
Equivalently, one can talk about properties of mutation classes (or the corresponding cluster algebras) which are preserved under restriction, i.e., taking any quiver in a mutation class, restricting to its full subquiver, then taking the mutation class of the latter. 
\end{definition}

\begin{example}
The following properties of a quiver~$Q$ are hereditary and mutation-invariant (see \cite[Section~4.1]{FWZ}, \cite[Section~5.2]{keller-survey}, and \cite[Section~3]{muller-max-green} for details and further references): 
\begin{itemize}[leftmargin=.3in]
\item 
$Q$ has \emph{finite type}, cf.\ \cite{ca2}, \cite[Remark~5.10.9]{FWZ}; 
\item 
$Q$ has \emph{finite mutation type}, cf.\ \cite[Proposition~4.1.4]{FWZ}; 
\item
$Q$ is a full subquiver of a particular (very special) infinite quiver, cf., e.g.,~\cite{Henrich}; 
\item 
$Q$ has \emph{finite/tame representation type}, cf.\ \cite[Theorem 5.5]{keller-survey}; 
\item
any quiver in $[Q]$ has arrow multiplicities $\le k$ (for some $k\in\ZZ_{>0}$), cf.\ \emph{loc.\ cit.}; 
\item 
$Q$ is \emph{mutation acyclic}, see Remark~\ref{rem:acyclic-not-3-universal} above; 
\item
$Q$ is \emph{embeddable} into a given mutation class, see \cite[Definition 4.1.8]{FWZ};  
\item 
$Q$ admits a \emph{reddening} (``green-to-red'') \emph{sequence} 
\cite[Theorem 17, Corollary~19]{muller-max-green}. 
\end{itemize}
We note that while the existence of a \emph{maximal green sequence} is a hereditary property, it is \underbar{not} mutation-invariant, see~\cite[Theorem~9, Corollary~14]{muller-max-green}. 

None of the above properties holds for all quivers. 
\end{example}

To illustrate the concepts introduced in Definition~\ref{def:hereditary}, 
we provide a curious application, cf.\ Remark~\ref{rem:grid-not-3-universal}. 

\begin{proposition}
\label{prop:grid-no-markov}
Let $Q$ be a quiver mutation equivalent to a grid quiver $A_k\boxtimes A_\ell$, for some $k$ and~$\ell$ (cf.\ Example~\ref{example:grid-quivers}). Then $Q$ cannot  have a full subquiver isomorphic to the Markov quiver, see Figure~\ref{fig:Markov-quiver}. In particular, every grid quiver is not $3$-universal. 
\end{proposition}

\begin{proof}
It is well known (cf.\ \cite[Section~1]{marsh-scott}, \cite[Section~5]{keller-maxgreen}) that any grid quiver possesses a maximal green sequence, and therefore a reddening sequence. As mentioned above, the property of having a reddening sequence is both mutation invariant and hereditary \cite[\emph{loc.\ cit.}]{muller-max-green}. Since the Markov quiver does not have this property, the claim follows. 
\end{proof}

The following simple observation relates hereditary properties to the notion of $n$-universality. 

\pagebreak[3]

\begin{remark}
\label{rem:enough-n-universal}
Let $P$ be a mutation-invariant hereditary property of quivers. 
If~some $n$-universal quiver has property~$P$, then any quiver on $\le n$~vertices has this property.
\end{remark}

In most applications, the number of vertices in a quiver is not fixed \emph{a~priori}. With this in mind, we introduce the following adaptation of the concept of universality: 

\begin{definition}
\label{def:universal-collection}
Let $\mathcal{Q}$ be a collection of quivers. We call $\mathcal{Q}$ a \emph{universal collection} if any quiver is a full subquiver of a quiver mutation equivalent to a quiver in~$\mathcal{Q}$. 
\end{definition}


\begin{remark}
\label{rem:universal-from-n-universal}
Let $Q_2, Q_3,\dots$ be a sequence of quivers such that each quiver $Q_n$~is \hbox{$n$-universal}. (For instance, let $Q_n$ be the $n$-universal quiver constructed in the proof of Theorem~\ref{thm:main-thm-quivers}.) Then $(Q_n)$ is a universal collection. 
\end{remark}


\begin{proposition}
\label{prop:universal-hereditary-alternative}
Let $\mathcal{Q}$ be a collection of quivers. Then exactly one of the following two possibilities holds:
\begin{itemize}[leftmargin=.3in]
\item
$\mathcal{Q}$ is a universal collection; 
\item
all the quivers in $\mathcal{Q}$ share a nontrivial mutation-invariant hereditary property.
\end{itemize}
(Here ``nontrivial'' means that not every quiver has this property.) 
\end{proposition}

\begin{proof}
Let us call a quiver $Q$ \emph{embeddable} into $\mathcal{Q}$ if $Q$ is a subquiver of a quiver mutation equivalent to a quiver in~$\mathcal{Q}$. There are two mutually exclusive possibilities:

\emph{Case~1}: any quiver is embeddable into~$\mathcal{Q}$. Then $\mathcal{Q}$ is universal, cf.\ Definition~\ref{def:universal-collection}. Furthermore, if the quivers in $\mathcal{Q}$ share a mutation-invariant hereditary property, then by Lemma~\ref{lem:enough-universal-collection}, any quiver has that property, so the property is trivial. 

\emph{Case~2}: some quiver~$Q$ is not embeddable into~$\mathcal{Q}$. Then $\mathcal{Q}$ is not universal. Furthermore, embeddability into~$\mathcal{Q}$ is a mutation-invariant and hereditary property which is both nontrivial (since $Q$ does not have it) and shared by all quivers~in~$\mathcal{Q}$. 
\end{proof}

\begin{lemma}
\label{lem:enough-universal-collection}
Let $P$ be a hereditary property of quivers, 
and let $\mathcal{Q}$ be a universal collection of quivers. 
 Then: 
\begin{itemize}[leftmargin=.3in]
\item
If any quiver mutation equivalent to a quiver in~$\mathcal{Q}$ has property~$P$, 
then any quiver has property~$P$.
\item
In particular, if $P$ is mutation-invariant, and all quivers in $\mathcal{Q}$ have property~$P$, 
then any quiver has property~$P$.
\end{itemize}
\end{lemma}

\begin{proof}
Since $\mathcal{Q}$ is universal, 
any quiver $Q$ is a full subquiver of a quiver~$Q'$ 
mutation equivalent to a quiver $Q''\in\mathcal{Q}$. 
Since $Q'$ has property~$P$, and $P$ is hereditary, it follows that $Q$ has property~$P$. 
\end{proof}
 
Suppose that we want to prove that a certain mutation-invariant hereditary property holds for all quivers. By Lemma~\ref{lem:enough-universal-collection}, it is enough to establish it for the quivers in some universal collection. 
Since the quivers in this collection may be fairly special, the latter claim may be amenable to proof techniques which are not available for arbitrary quivers. 


We illustrate this idea by proposing a hypothetical strategy for the proof of \emph{Laurent positivity} for cluster algebras. 
%
Let $Q$ be a quiver without frozen vertices, and $\mathcal{A}(Q)$ the associated cluster algebra. (We refer the reader to~\cite{FWZ} for basic background on cluster algebras.) 
The following strengthening of the Laurent Phenomenon for cluster algebras 
was conjectured in~\cite{ca1} and proved in~\cite{GHKK,lee-schiffler}:  
\begin{itemize}
\item[{\rm (LP+)}]
any cluster variable~$z$ in~$\mathcal{A}(Q)$ is expressed in terms of any  cluster~$\mathbf{x}$ of this cluster algebra as a Laurent polynomial with positive coefficients. 
\end{itemize}
No~simple proof of this result is known. As a potential alternative to the highly technical proofs given in~\cite{GHKK} and~\cite{lee-schiffler}, we suggest an approach based on the following elementary result (which does not rely on~\cite{GHKK,lee-schiffler}). 

\begin{lemma}
\label{lem:laurent-positivity-is-hereditary}
The statement {\rm (LP+)} above, viewed as a property of a quiver~$Q$, is both mutation-invariant and hereditary. 
\end{lemma}

\begin{proof}
Mutation invariance is immediate by construction. Let us prove that (LP+) is  hereditary. Given a subset of indices~$I$, consider the restricted quiver~$Q_I$ and the corresponding cluster algebra~$\mathcal{A}(Q_I)$. We need to verify that if (LP+) holds for~$\mathcal{A}(Q)$, then it also holds for~$\mathcal{A}(Q_I)$. 

Let $\overline I$ denote the complement of~$I$ in the set of all indices. The labeled seed pattern for~$\mathcal{A}(Q_I)$ is obtained from the labeled seed pattern for~$\mathcal{A}(Q)$ as follows: 
\begin{itemize}[leftmargin=.3in]
\item[(a)]
freeze all indices in~$\overline I$, thereby allowing only mutations~$\mu_i$ in directions $i\in I$; this restricts the seed pattern to a subpattern supported on a smaller regular tree; 
\item[(b)]
trivialize all coefficient variables~$x_i$ ($i\in\overline I$), i.e., specialize $x_i=1$;
\item[(c)]
apply the same specialization as in (b) to each cluster variable in the pattern, viewed as a rational function in some base cluster; the result is well defined, and does not depend on which base cluster we use. 
\end{itemize}
Let $z$ be a cluster variable in the labeled seed pattern for~$\mathcal{A}(Q_I)$, and let $\tilde z$ be the corresponding cluster variable in~$\mathcal{A}(Q)$. Similarly, let~$\mathbf{x}$ be a cluster  for~$\mathcal{A}(Q_I)$, and let~$\mathbf{\tilde x}$ be the corresponding cluster for~$\mathcal{A}(Q)$. 
Then~$z$ (resp.,~$\mathbf{x}$) is obtained from~$\tilde z$ (resp.,~$\mathbf{\tilde x}$) via the specialization described in~(b) above. 
Hence the rational function that expresses $z$ in terms of~$\mathbf{x}$ is obtained from the rational expression for $\tilde z$ in terms of~$\mathbf{\tilde x}$ via the same specialization. If the latter expression is a Laurent polynomial with positive coefficients, then so is the former. Thus the property (LP+) is hereditary. 
\end{proof}

\begin{observation}
By Lemma~\ref{lem:laurent-positivity-is-hereditary}, once we establish the Laurent positivity property (LP+) for a particular universal collection of quivers (say a certain sequence of $n$-universal quivers, cf.\ Remark~\ref{rem:universal-from-n-universal}), (LP+) in full generality follows. 
\end{observation}


Another important 
hereditary property is the sign-coherence of $c$-vectors. 
Let us quickly recall the relevant definitions. 
(See~\cite{nakanishi-zelevinsky} and~\cite{ca4, keller-fpsac13} for additional details.)

\begin{definition}
Let $Q$ be a quiver on an $n$-element vertex set~$V$. 
The corresponding \emph{framed quiver} $\widetilde Q$ is a quiver on $2n$ vertices
obtained from $Q$ as follows. 
The vertex set of $\widetilde Q$ is the set $V\sqcup \overline V$ where 
$\overline V=\{\bar v:v\in V\}$. The arrows in $\widetilde Q$ are the ones in~$Q$ plus 
$n$ new arrows of the form $v\to \bar v$, for $v\in V$. 
See Figure~\ref{fig:2-universal-quiver-framed}. 
\end{definition}

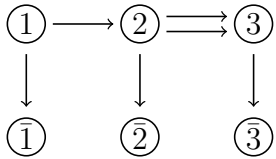
\begin{figure}[ht]
\begin{center}
\begin{tikzpicture}[scale=1,line width=0.6pt]
\node at (1.75,0) {1};
\draw (1.75,0) circle [radius=0.25];
\draw[->](2.1,0)--(2.9,0);
\node at (3.25,0) {2};
\draw (3.25,0) circle [radius=0.25];
\draw[->](3.6,0.1)--(4.4,0.1);
\draw[->](3.6,-0.1)--(4.4,-0.1);
\node at (4.75,0) {3};
\draw (4.75,0) circle [radius=0.25];
\node at (1.75,-1.5) {$\bar 1$};
\draw (1.75,-1.5) circle [radius=0.25];
\draw[<-](1.75,-1.1)--(1.75,-0.4);
\node at (3.25,-1.5) {$\bar 2$};
\draw (3.25,-1.5) circle [radius=0.25];
\draw[<-](3.25,-1.1)--(3.25,-0.4);
\node at (4.75,-1.5) {$\bar 3$};
\draw (4.75,-1.5) circle [radius=0.25];
\draw[<-](4.75,-1.1)--(4.75,-0.4);
\end{tikzpicture}
\end{center}
\caption{The framed quiver $\widetilde Q$ for the $2$-universal quiver~$Q$ in Figure~\ref{fig:2-universal-quiver}. }
\vspace{-.2in}
\label{fig:2-universal-quiver-framed}
\end{figure}

The following \emph{sign coherence} property was implicitly conjectured in~\cite{ca4}, and proved in~\cite{DWZ2}
(other proofs have been given in \cite{GHKK, Nagao}):
\begin{itemize}
\item[{\rm (SC)}]
in any quiver $\widetilde Q'$ obtained from~$\widetilde Q$ by a sequence of mutations
at the vertices in~$V$, 
no two vertices from the set $\overline V$ are connected by an arrow. 
\end{itemize}

\pagebreak[3]

\begin{remark}
The conventional version of (SC) asserts that $\widetilde Q'$ does not contain any 
arrows of the form $\bar u\to v\to \bar w$. 
(Also, the vertices in~$\overline V$ are declared frozen.) 
These two versions are equivalent to each~other: 
if $\widetilde Q'$ contains a configuration $\bar u\to v\to \bar w$,
then $\mu_v(\widetilde Q')$ has an arrow $\bar u\to \bar w$; conversely, 
the only way such an arrow could arise is via a configuration of this kind. 
\end{remark}

\begin{remark}
The statement (SC) is usually called ``sign coherence of $c$-vectors.''
There is a companion (arguably more important) statement called
``sign coherence of $\mathbf{g}$-vectors''~\cite{ca4}.  
As observed in~\cite{nakanishi-zelevinsky}, the latter easily follows from the former. 
\end{remark}

In spite of its elementary formulation, no simple proof of sign-coherence is presently known. 
All known proofs are highly nontrivial, 
and go far beyond the elementary combinatorics of quiver mutations. 
(The proof in \cite{DWZ2} relies on the existence of a nondegenerate potential for any quiver,
whereas the proof in \cite{GHKK} involves (coherent) scattering diagrams.) 

By contrast, the following fact is very easy to verify. 

\begin{lemma}
\label{lem:sc-hereditary}
The statement {\rm (SC)}, viewed as a property of a quiver~$Q$, is hereditary. 
\end{lemma}

\begin{proof}
Let $Q_2$ be a full subquiver of a quiver~$Q_1$.
Then the framed quiver $\widetilde Q_2$ is a full subquiver of~$\widetilde Q_1$.
Furthermore, any quiver~$\widetilde Q'_2$ obtained from $\widetilde Q_2$ 
via mutations at the vertices of~$Q_2$
is a full subquiver of the quiver~$\widetilde Q'_1$ obtained from $\widetilde Q_1$ via the same sequence of mutations.
Thus, if $\widetilde Q'_1$ does not contain arrows of the form $\bar u\to \bar v$,
then the same holds for~$\widetilde Q'_2$.  
\end{proof}

$\!$Lemma~\ref{lem:sc-hereditary} suggests a potential strategy 
towards a simpler proof of sign-coherence:
\begin{itemize}[leftmargin=.3in]
\item[(i)]
show (ideally, by a direct combinatorial argument) that the sign-coherence property~(SC), 
or some hereditary \linebreak[3] strengthening thereof, is mutation invariant; 
\item[(ii)]
show that this property holds for a particular universal collection of quivers. 
\end{itemize}
Unfortunately, we have been unable to successfully implement this strategy.

\section{Additive categorification}
\label{sec:categorification}

Roughly speaking, an \emph{(additive) categorification} of a quiver 
$Q$ is a diagram in the shape of $Q$ in an abelian or triangulated category together with a rule for mutating the diagram as the quiver undergoes a mutation. 
There is extensive literature on this popular topic, mostly in the case when the quivers are acyclic. 
(There is also an alternative notion of a \emph{monoidal categorification}~\cite{leclerc-hernandez}
which we do not discuss here.) 


Existence of a categorification is a hereditary property: a categorification of~$Q$ gives rise to a categorification of any subquiver of a quiver mutation equivalent~to~$Q$. 
There are known examples of quivers which cannot be categorified in the traditional sense. 
Consequently, any $n$-universal quiver (for sufficiently large~$n$) cannot be categorified. 

In this section, we give a simple direct argument showing why universal quivers admit no Hom-finite categorification. We also note that universal quivers serve as a natural test case for any attempt to generalize the notion of categorification using Hom-infinite Jacobian algebras, as suggested by P.-G.~Plamondon~\cite{plamondon}.


Let us quickly review the basic ideas, sticking to the triangulated category setting. 
We will need a more general notion of a quiver which involves designating some of its vertices as \emph{frozen}. (Mutations at those vertices are forbidden.) 

From now on, we fix a field $K$ and work with $K$-categories. These are categories $\mathcal C$ in which $\Hom(X,Y)$ is a vector space over~$K$ for any objects $X,Y\in \mathcal C$, and composition is $K$-bilinear. We say that $\mathcal C$ is \emph{Hom-finite} if $\Hom(X,Y)$ is finite dimensional for any $X,Y\in \mathcal C$.
Any object $X=\bigoplus_{i=1}^nX_i$ in a $\operatorname{Hom}$-finite Krull-Schmidt category defines a finite-dimensional algebra $\operatorname{End}(X)$.

\begin{definition}[{\cite[Section~II.1]{BIRSc}}]
\label{def:BIRSc}
Let $\mathcal C$ be a $\Hom$-finite, Krull-Schmidt, 2-Calabi-Yau triangulated $K$-category.  
That is, all objects are direct sums of indecomposables, with canonical isomorphisms 
$\Hom(X,Y)\!\cong \!D\!\operatorname{Hom}(Y,X[2])$
where $D\!=\!\operatorname{Hom}_K(-,K)$ denotes the vector space dual. 
A \emph{cluster structure} on~$\mathcal C$ is a collection of 
objects---called \emph{extended clusters}---some of whose components are designated  \emph{frozen}; this collection must satisfy the following conditions:
\begin{itemize}[leftmargin=.3in]
\item[(a)]
For every nonfrozen component $T_k$ of an extended cluster $T$, there is a unique, up to isomorphism, indecomposable object $T_k'$ such that $T'=T/T_k\oplus T_k'$ is another extended cluster. 
Moreover $T_k$ and $T_k'$ are related by two distinguished triangles
\begin{align*}
T_k[-1]\to T_k'\to B\to &T_k \,, \\
&T_k\to B'\to T_k'\to T_k[1] \,;
\end{align*}
here $B$ (resp.~$B'$) is the minimal right (resp.~left) $\operatorname{add}(T/T_k)$-approximation~of~$T_k$.
We say that $T'$ is obtained from $T$ by a \emph{mutation} at~$T_k$. 


\item[(b)]
For every extended cluster $T$, define the quiver $Q_T$ as follows. 
The vertices~$v_k$ of $Q_T$ are in bijection with the components~$T_k$ of~$T$.
For each pair of vertices $v_i$ and~$v_j$, 
at least one of them nonfrozen, 
draw  $\dim(\Hom(T_i,T_j)/I_{ij})$ arrows $v_i\to v_j$
where $I_{ij}$ is generated by the morphisms which factor through the remaining objects~$T_k$,
for $k\notin\{i,j\}$. 
The quiver $Q_T$ constructed in this way should have no loops and no 2-cycles.  

\item[(c)]
The quiver $Q_{T'}$ of $T'=T/T_k\oplus T_k'$ is obtained from $Q_T$ by mutating at~$v_k$.
\end{itemize}
When conditions (a)--(c) are satisfied, the resulting cluster structure on the category~$\mathcal{C}$
is called a (Hom-finite, additive) \emph{categorification} of any of the quivers~$Q_T$.

Any such categorification can be restricted to the set of extended clusters which are \emph{reachable} from a given extended cluster~$T$ by a sequence of mutations. 
These reachable clusters naturally correspond to the seeds of the cluster algebra defined by the quiver~$Q_T$, and 
the quivers appearing in the restricted categorification 
are precisely the quivers mutation equivalent to~$Q_T$.  
\end{definition}

\pagebreak[3]


\begin{theorem}
\label{th:universal-no-categorification}
Let $Q$ be one of the quivers shown in Figures~\ref{fig:extended-somos-4}--\ref{fig:double-4-cycle}. 
Then:
\begin{itemize}[leftmargin=.3in]
\item
the quiver $Q$ does not allow a Hom-finite categorification; 
\item
for any potential $S$ for $Q$, the Jacobian algebra $J(Q,S)$ is infinite-dimensional.
\end{itemize}
\end{theorem}

\begin{proof}
In a categorification of~$Q$, the initial cluster consists of objects corresponding to the vertices of~$Q$. When we mutate at the vertices other than $u$ and~$v$ (see Figures~\ref{fig:extended-somos-4}--\ref{fig:double-4-cycle}), the objects $T_u$ and $T_v$ corresponding to $u$ and~$v$ remain unchanged. After any such sequence of mutations, the number of arrows from $u$ to $v$ equals the dimension of a certain quotient of $\Hom(T_u,T_v)$, see Definition~\ref{def:BIRSc}(b). 
In particular, this number (which depends on the mutation sequence) cannot exceed $\dim\Hom(T_u,T_v)$. 
Since this number is unbounded, $\Hom(T_u,T_v)$ must be infinite-dimensional, so the categorification is not Hom-finite. 

Suppose $J(Q,S)$ is finite-dimensional. Under this condition, 
C.~Amiot~\cite{amiot} constructed a cluster category~$\mathcal C(Q,S)$,
and showed that $\mathcal C(Q,S)$ is a Hom-finite categorification of the cluster algebra corresponding to~$Q$.
Since such a Hom-finite categorification does not exist, we conclude that $\dim J(Q,S)=\infty$. 
\end{proof}

The conclusions of Theorem~\ref{th:universal-no-categorification} hold for any quiver~$Q$ satisfying the conditions in Lemma~\ref{lem:gluing}, and for any quiver which is mutation equivalent to a quiver containing such~$Q$ as a full subquiver.

\end{document}